\documentclass{aptpub}

\authornames{C. Geiss, C. Labart, A. Luoto} 
\shorttitle{Random walk approximation of FBSDEs} 

\usepackage{a4wide}
 \usepackage{amsmath, enumerate} 
 \usepackage[bookmarks,colorlinks,linkcolor=blue]{hyperref}
\usepackage[usenames,dvipsnames,svgnames,table]{xcolor}

\newcommand{\rset}{\mathbb{R}}
 \newcommand{\R}{\mathbb{R}}
 \newcommand{\PP}{\mathbb{P}}

\newcommand{\ep}{\varepsilon}
\newcommand{\ind}{\mathbf{1}}

\newcommand{\e}{\mathbb{E}}

\newcommand{\p}{\mathbb{P}}

 \newcommand{\cf}{\mathcal{F}}
\newcommand{\cg}{\mathcal{G}}

\newcommand{\sX}{{\sf X}}
\newcommand{\sY}{{\sf Y}}
\newcommand{\sZ}{{\sf Z}}
\newcommand{\sff}{{\sf f}}

\newcommand{\eps}{\varepsilon}

\newcommand{\nb}{\nabla}
\newcommand{\cX}{\mathcal{X}}

\newcommand{\ti}{\tilde}

\newcommand{\les}{\hspace{-2em}}
\newcommand{\equa}{\begin{eqnarray*}}
\newcommand{\tion}{\end{eqnarray*}}
\newcommand{\equal}{\begin{eqnarray}}
\newcommand{\tionl}{\end{eqnarray}}
\newcommand{\pull}{\hspace{-2em}}
\newcommand{\half}{\frac{1}{2}}

\newcommand{\pr}[1]{\left( #1 \right)}

\newcommand{\cG}{\mathcal{G}}
\newcommand{\cD}{\mathcal{D}}

\newcommand{\less}{\hspace*{-4em}}

\def\norm#1{\|#1\|}

 \newtheorem{lemme}{Lemma}[section]
\newtheorem{hypo}{Assumption}[section]
\newtheorem{defi}{Definition}[section]

\newcommand{\ce}{}

\begin{document}

\title{Mean square rate of convergence for random walk approximation of 
forward-backward SDEs} 

\authorone[Department of Mathematics and Statistics, University of Jyvaskyla]{Christel Geiss} 
\authortwo[Univ. Grenoble Alpes, Univ. Savoie Mont Blanc, CNRS, LAMA]{C{\'e}line Labart}
\authorone[Department of Mathematics and Statistics, University of Jyvaskyla]{Antti Luoto}
\addressone{P.O.Box 35 (MaD) FI-40014 University of Jyvaskyla, Finland}
\addresstwo{Univ. Grenoble Alpes, Univ. Savoie Mont Blanc, CNRS, LAMA,
  73000 Chamb\'ery, France} 

\begin{abstract}
Let $(Y, Z)$ denote the solution to a forward-backward SDE. If one
  constructs a random walk $B^n$ from the underlying Brownian motion $B$ by
  Skorohod embedding, one can show $L_2$-convergence of the
  corresponding solutions $(Y^n,Z^n)$ to $(Y, Z).$
  We estimate the rate of convergence  in dependence of smoothness properties, especially for a  terminal condition function in $C^{2,\alpha}$.\\
  The proof relies on an approximative representation of $Z^n$ and uses the
  concept of discretized Malliavin calculus.  Moreover, we use growth and
  smoothness properties of the PDE associated to the FBSDE as well as of the
  finite difference equations associated to the approximating stochastic
  equations. We derive these properties by probabilistic methods.
\end{abstract}

\keywords{Backward stochastic differential equations; approximation
scheme; finite difference equation; convergence rate;  random walk approximation} 

\ams{60H10; 60H35; 60G50}{60H30} 

\section{Introduction} 

Let $(\Omega, \cf, \PP)$ be a complete probability space carrying the standard
Brownian motion $B= (B_t)_{t \ge 0}$ and assume that $(\cf_t)_{t\ge 0}$
is the augmented natural filtration.  Let $(Y, Z)$ be the solution of the
forward-backward SDE (FBSDE)
\begin{align}\label{BSDE}
X_s&= x + \int_0^s b(r,X_r) dr + \int_0^s \sigma(r,X_r) dB_r,\notag \\
Y_s&= g( X_T) + \int_s^T f(r,X_r,Y_r,Z_r)dr - \int_s^T Z_r dB_r,  \quad \quad 0 \leq s  \leq T.
\end{align}
Let $(Y^n,Z^n)$ be the solution of the FBSDE if the Brownian motion $B$ is
replaced by a scaled random walk $B^n$ given by \equal \label{Bn} B^n_{t} =
\sqrt{h} \sum_{i=1}^{[t/h] }\varepsilon_i, \quad \quad 0 \leq t \leq T, \tionl
where $h= \tfrac{T}{n}$ and $(\varepsilon_i)_{i=1,2, \dots}$ is a sequence of
i.i.d.~Rademacher random variables. Then $(Y^n,Z^n)$ solves the discretized
FBSDE
\begin{align}\label{discrete-BSDE}
X^n_s &= x + \int_{(0,s]} b(r,X^n_{r-}) d[B^n]_r + \int_{(0,s]} \sigma(r,X^n_{r-})
         dB^n_r, 
  \notag \\
    Y^n_s &= g(X^n_T) + \int_{(s,T]} f(r,X^n_{r-}Y^n_{r-},Z^n_{r-})d[B^n]_r - \int_{(s,T]} Z^n _{r-} dB^n_r, \quad 0  \le s \leq T.
\end{align}

The approximation of BSDEs using random walk has been investigated by many
authors, also numerically (see, for example, \cite{Briand-etal},
\cite{JanczakII}, \cite{MaProMartTor}, \cite{MarMarTor}, \cite{MeminPengXu},
\cite{PengXu}, \cite{Cher_and_stadje}). In 2001, Briand et al.
\cite{Briand-etal} have shown weak convergence of $(Y^n,Z^n)$ to $(Y,Z)$ for a
Lipschitz continuous generator $f$ and a terminal condition in $L_2.$ The rate
of convergence of this method remained an open problem.
\smallskip
Bouchard and Touzi in \cite{BT04} and Zhang in \cite{Z04} proposed instead of
random walk an approach based on the dynamic programming equation, for which
they established a rate of convergence. But this approach involves conditional
expectations.  Various methods to approximate these conditional expectations
have been developed (\cite{GLW05}, \cite{CMT10}, \cite{CGT17}). Also forward
methods have been introduced to approximate \eqref{BSDE}: a branching
diffusion method (\cite{HLTT13}), a multilevel Picard approximation
(\cite{WHJK17}) and Wiener chaos expansion (\cite{BL14}).  Many extensions of
\eqref{BSDE} have been considered, among them schemes for reflected BSDEs
(\cite{BP03}, \cite{CR16}), high order schemes (\cite{Cha14}, \cite{CC14}),
fully-coupled BSDEs (\cite{Delarue_and_Menozzi}, \cite{BZ08}), quadratic BSDEs
(\cite{CR15}), BSDEs with jumps (\cite{GL16}) and McKean-Vlasov BSDEs
(\cite{A15}, \cite{CT15}, \cite{CCD17}).\\

The aim of this paper is to study the rate of the $L_2$-approximation of
$(Y^n_t,Z^n_t)$ to $(Y_t,Z_t)$ when $X$ satisfies \eqref{BSDE}. For this,
  we generate the random walk $B^n$ by Skorohod embedding from the Brownian
  motion $B$.  In this case the $L_p$-convergence  of $B^n$ to $B$  is of order
  $h^{\frac{1}{4}}$ for any $p>0.$ 
  The special case $X=B$ has already been studied in
\cite{GLL},  assuming a locally $\alpha$-H\"older continuous terminal
function $g$ and a Lipschitz continuous generator. An estimate for the rate of convergence was obtained which is of
order $h^{\frac{\alpha}{4}}$ for the $L_2$-norm of $Y^n_t-Y_t,$
and of order
$\tfrac{h^{\frac{\alpha}{4}}}{\sqrt{T-t}}$ for the $L_2$-norm of $Z^n_t-Z_t.$

In the present paper, where we
assume that $X$ is a solution of the SDE in \eqref{BSDE}, rather strong
conditions on the smoothness and boundedness on $f$ and $g$ and also on $b$
and $\sigma$  are needed.  In Theorem \ref{the-result}, the main result
of the paper, we show that the convergence rate for $(Y^n_t,Z^n_t)$ to
$(Y_t,Z_t)$ in $L_2$ is of order $h^{\frac{1}{4}\wedge \frac{\alpha}{2}}$
provided that $g''$ is locally $\alpha$-H\"older continuous. To the best of
our knowledge, these are the first cases a convergence rate for the
approximation of forward-backward SDEs using random walk has been obtained.

  \begin{rem}For the diffusion setting -- in contrast to the case $X=B$ -- we can derive
  the convergence rate for $(Y^n_t,Z^n_t)$ to $(Y_t,Z_t)$ in $L_2$ only under
  strong smoothness conditions on the coefficients which include also that
  $g''$ is locally $\alpha$-H\"older continuous (see Assumption \ref{Xgeneral}
  below). These requirements appear to be necessary.  This becomes visible in
  Subsection \ref{sect:Zn} where we introduce a discretized Malliavin weight
  to obtain a representation $\hat Z^n$ for $Z^n.$ While it holds that
  $\hat Z^n =Z^n$ when $X=B,$ in our case $\hat Z^n$ does not coincide with
  $Z^n.$ However, one can show that the difference $\hat Z^n_t -Z^n_t$
  converges to $0$ in $L_2$ as $n \to \infty$ using a H\"older continuity
  property (see \eqref{D-Fx-difference} in Remark \ref{F-n-property}) for the
  space derivative of the generator in \eqref{discrete-BSDE}. For this
  H\"older continuity property to hold one needs enough smoothness in space
  from the solution $u^n$ to the finite difference equation associated to the
  discretized FBSDE \eqref{discrete-BSDE}.  Provided that Assumption
  \ref{Xgeneral} holds we show the smoothness properties for $u^n$ in
  Proposition \ref{u-discrete} applying methods known for L\'evy driven
  BSDEs.
\end{rem}

The paper is organized as follows: Section \ref{2} contains the setting, main
assumptions and the approximative representation $\hat{Z}^n$ of $Z^n.$ Our
main results about the approximation rate for the case of no generator
(i.e. $f=0$) and for the general case are in Section \ref{3}. One can see that
in contrast to what is known for time discretization schemes, for random walk
schemes the Lipschitz generator seems to cause more difficulties than the
terminal condition: while in the case $f=0$ we need that $g'$ is locally
$\alpha$-H\"older continuous, in the case  $f\neq0$ is
this property is required for $g''.$  In Section \ref{4} we recall some
needed facts about Malliavin weights, about the regularity of solutions to
BSDEs and properties of the associated PDEs. Finally, we sketch how to prove
growth and smoothness properties of solutions to the finite difference
equation associated to the discretized FBSDE.  Section \ref{5} contains
technical results which mainly arise from the fact that the construction of
the random walk by Skorohod embedding forces us to compare our processes on
different 'time lines', one coming from the stopping times of the Skorohod
embedding, and the other one is ruled by the equidistant deterministic times
due to the quadratic variation process $[B^n].$
 
 \section{Preliminaries} \label{2}
\subsection{The SDE and its approximation scheme}

We introduce 
 \begin{align*}
X_t = x + \int_0^t b(s,X_s) ds + \int_0^t \sigma(s,X_s) dB_s, \quad \quad 0 \le t \le T
\end{align*}
and its discretized counterpart
\begin{align} \label{Xn} 
X^n_{t_k} = x + h\sum_{j=1}^k b(t_j,X^n_{t_{j-1}})  + \sqrt{h} \sum_{j=1}^k \sigma(t_j,X^n_{t_{j-1}}) \ep_j, \quad \quad t_j:=j\tfrac{T}{n}, \,\, j=0,...,n,
\end{align}
where $(\varepsilon_i)_{i=1,2, \dots}$ is a sequence of i.i.d.~Rademacher random variables. 
Letting $\mathcal{G}_k := \sigma(\varepsilon_i: 1 \leq i \leq k)$ with $\mathcal{G}_0 :=\{\emptyset, \Omega\},$  it follows that the associated discrete-time random walk $(B^n_{t_k})_{k=0}^n$ is $(\mathcal{G}_k)_{k=0}^n$-adapted. Recall \eqref{Bn} and  $h= \tfrac{T}{n}.$ 
If we extend  the sequence $(X^n_{t_k})_{k\ge0}$ to a process in continuous time by defining  $X^n_t := X^n_{t_k} $ for $t \in [t_k, t_{k+1}),$ 
it is the  solution of the  forward SDE  \eqref{discrete-BSDE}. 

 We formulate our first assumptions. Assumption  \ref{hypo2} \eqref{b} will be not used  explicitely for our estimates  but it is required for Theorem \ref{difference-estimates for Y and Z} below.

\begin{hypo}\label{hypo2}\hfill
  \begin{enumerate}[(i)]
\item   $b, \sigma  \in C_b^{0,2}([0,T]\times \R),$  in the sense that the  derivatives of order $k=0,1,2$  w.r.t.~the space variable  are continuous and bounded  on  $[0,T]\times \R,$
\item  \label{b}  the  first and second derivatives of $b$ and $\sigma$  w.r.t.~the space variable are assumed to be $\gamma$-H\"older continuous 
(for some $\gamma \in (0,1],$ w.r.t. the parabolic metric
 $d((t,x),(\bar t, \bar x))=(|t- \bar t| + |x- \bar x|^2)^\half)$ on all compact subsets of $[0, T ] \times \R.$ 
\item   $b, \sigma$ are $\half$-H\"older continuous in time,
uniformly in space, 
\item $\sigma(t,x) \ge \delta >0$ for all $(t,x).$  
\end{enumerate}
\end{hypo}

\begin{hypo}\label{hypo3}\hfill
\begin{enumerate}[(i)]
\item
   $g$ is locally H\"older continuous with order $\alpha \in  (0,1]$ and polynomially bounded ($p_0 \ge 0, C_g>0$)  in the following sense
\begin{align} \label{eq5}
 \forall (x, \bar x) \in \R^2, \quad |g(x)-g( \bar x)|\le C_g(1+|x|^{p_0}+ | \bar  x|^{p_0})|x- \bar x|^{\alpha}.
\end{align}
 \item The function   $ [0,T] \times \R^3: (t,x,y,z) \mapsto f (t,x,y,z)$ satisfies 
   \equal  \label{Lipschitz}
 |f(t,x,y,z) - f(\bar t, \bar x, \bar y, \bar z)| \le L_f(\sqrt{t- \bar  t} + |x- \bar x| + |y- \bar y|+|z- \bar z|). 
  \tionl   
\end{enumerate}
\end{hypo}

 Notice that \eqref{eq5} implies
\equal \label{Psi}
|g(x)| \le K(1+|x|^{p_0+1}) =:\Psi(x), \quad x \in \R,
\tionl
for some $K>0.$  From the continuity of $f$ we conclude that 
$$ 
    K_f:=\sup_{0\le t\le T} |f(t,0,0,0)|  < \infty.
$$

\noindent {\bf Notation:} 
\begin{itemize}
\item $\|\cdot\|_p:=\|\cdot\|_{L_p(\p)}$ for $p\ge 1$ and for $p=2$ simply
  $\|\cdot\|$.
\item If  $a$ is a function, $C(a)$ represents a generic constant which 
  depends on $a$ and possibly also on its derivatives.
\item $\e_{0,x}:=\e(\cdot|X_0=x)$.
\item  Let $\phi$ be a $C^{0,1}([0,T]\times \R)$ function. $\phi_x$ denotes
  $\partial_x \phi$, the
  partial derivative of $\phi$ w.r.t. $x$.
\end{itemize}

\subsection{The FBSDE and its approximation scheme}
 
 Recall  the FBSDE  \eqref{BSDE} and  its approximation  \eqref{discrete-BSDE}. The backward equation in \eqref{discrete-BSDE} can equivalently 
be written  in the form
\begin{align}\label{discrete-BSDE-sums}
Y^n_{t_k} &= g(X^n_T) + h\sum_{m=k}^{n-1} f(t_{m+1},X^n_{t_m},Y^n_{t_m},Z^n_{t_m}) - \sqrt{h}\sum_{m=k}^{n-1}Z^n_{t_m} \ep_{m+1}, \quad 0 \leq k \leq n,\end{align}
if one puts   $X^n_r:=X^n_{t_m},$ \, $Y^n_r:=Y^n_{t_m}$  and $Z_r^n :={Z}_{t_m}^n$ for $r \in [t_m, t_{m+1}).$ 

 \begin{rem}
  Equations \eqref{discrete-BSDE} and \eqref{discrete-BSDE-sums} do not contain any orthogonal part to the random walk $B^n$
since we are in a special case where the orthogonal part is zero.  Indeed, for $(\varepsilon_k)_{k=1,\cdots,n}$ following the Rademacher law  assume
($\mathcal{G}_k:=\sigma(\varepsilon_i,i=1,\cdots,k)$) as filtration,
and  let for the  $\mathcal{G}_n$-measurable random variable  $F(\varepsilon_1, ..., \varepsilon_n)$ hold the representation 
$$F(\varepsilon_1, ... ,\varepsilon_n) = c +  \sum_{m=1}^{n} h_m \varepsilon_m + N_n,$$
where  $(h_m)_{m=1}^n$ is predictable and $(N_m)_{m=1}^n$ a martingale orthogonal to $(B^n_{t_m})_{m=1}^n$ given by $B^n_{t_m}  = \sqrt{h} (\varepsilon_1+...+\varepsilon_m).$   By definition,  orthogonality of the martingales $N$ and $B^n$ means that their product is a martingale, i.e. we have
 $$ \e[ N_{k+1} B^n_{t_{k+1}} |\mathcal{G}_k]= N_k B^n_{t_{k}}, $$
and since $N_k B^n_{t_{k}} =  \e[ N_{k+1}B^n_{t_{k}} |\mathcal{G}_k],$ this implies   especially that $\e N_{k+1}\varepsilon_{k+1}=0.$  Assume $N_{k+1}$ is given by $N_{k+1}= H(\varepsilon_1,...,\varepsilon_{k+1}).$
Then $0 = \e H(\varepsilon_1,...,\varepsilon_{k+1})\varepsilon_{k+1} =  \tfrac{1}{2}[H(\varepsilon_1,...,\varepsilon_{k},1) -H(\varepsilon_1,...,\varepsilon_{k},-1)]  $ 
implying that $N_{k+1}$ is $\mathcal{G}_k$-measurable since 
$ H(\varepsilon_1,...,\varepsilon_{k},1) =H(\varepsilon_1,...,\varepsilon_{k},-1),$
and
therefore the martingale  $(N_m)_{m=1}^n$ is identically zero. (See also \cite[page 3]{Briand-etal} or \cite[Proposition1.7.5]{Privault}.) 
\end{rem}

One can derive an equation  for $Z^n = (Z^n_{t_k})_{k=0}^{n-1}$ 
if one multiplies \eqref{discrete-BSDE-sums} by $\ep_{k+1}$  and takes the
conditional expectation w.r.t. $\cG_k$,  so that 
\equal  \label{cZ-exactly-1}
     Z^n_{t_k}     
 &=&   \frac{\e^{\cg}_{k}\left ( g(X^n_T) \ep_{k+1}\right )}{\sqrt{h}}  +\e^{\cg}_{k}\left (
   \sqrt{ h}\sum_{m=k+1}^{n-1}f( t_{m+1},X^n_{t_m}, Y^n_{t_m} , Z^n_{t_m})
   \ep_{k+1}\right ), \, 0\le  k \le n-1,
 \tionl
where $\e^{\cg}_k:=\e(\cdot|\cg_k)$.

 \begin{rem}
For $n$ large enough,  the BSDE \eqref{discrete-BSDE} has a unique solution $({Y}^n,
{Z}^n)$ (see \cite[Proposition 1.2]{Toldo}), and  $({Y}^n_{t_k}, {Z}^n_{t_k})_{k=0}^{n-1}$  is   adapted to the
filtration $(\cG_k)_{k=0}^{n-1}.$
\end{rem}

\subsubsection{Representation for \texorpdfstring{$Z$}{Z}} \label{sect:Z}

We will use the following representation for $Z$, due to Ma and Zhang (see \cite[Theorem 4.2]{MaZhang})
\begin{align}\label{Z-representation-1}
   Z_t &= \e_{t} \left( g(X_T) N^t_T + \int_t^T f(s,X_s,Y_s,Z_s) N^t_s
     ds\right) \sigma(t,X_t), \quad 0 \leq t \leq T
\end{align}
where  $\e_t :=\e(\cdot|\cf_t),$ and for all $s \in (t,T]$, we have (cf. Lemma \ref{Malliavin-weights})
\equal \label{N-cont}
  N^t_s=\frac{1}{s-t}\int_t^s\frac{\nabla X_r}{\sigma(r,X_r)\nabla X_t}dB_r,  \tionl
where $\nabla X = (\nabla X_s)_{s \in [0,T]}$ is the variational process i.e. it solves 
\begin{align}  \label{nabla-X}
  \nabla X_s = 1+\int_0^s b_x(r,X_r) \nabla X_r dr + \int_0^s \sigma_x(r,X_r) \nabla X_r dB_r,
 \end{align}
with $(X_s)_{s \in [0,T]}$ given in \eqref{BSDE}. 

\begin{rem}
  In the following we will assume that $g''$ exists. In such a case we have
  the following representation for $Z$:
  \begin{align}\label{Z-representation}
   Z_t &= \e_{t} \left(  g'(X_T) \nabla X_T + \int_t^T f(s,X_s,Y_s,Z_s) N^t_s
     ds\right) \sigma(t,X_t), \quad 0 \leq t \leq T.
\end{align}
\end{rem}
\subsubsection{Approximation for \texorpdfstring{$Z^n$}{Zn} } \label{sect:Zn}
In this section we state the discrete counterpart to \eqref{Z-representation-1},
which, in the general case of a forward process $X$, does not coincide with
$Z^n$ (given by \eqref{cZ-exactly-1}). In contrast to the continuous-time
case, where the variational process and the Malliavin derivative are connected
by $\tfrac{\nabla X_t}{\nabla X_s} = \tfrac{D_sX_t }{\sigma(s,X_s)}$
($s \le t$), we can not expect equality for the corresponding expressions if
we use the discretized version of the processes $(\nabla X_t)_t$
and $(D_s X_t)_{s\le t}$ introduced in \eqref{D^n_k-representation}. This
counterpart $\hat{Z}^n$ to $Z$ is a key tool in the proof of the convergence
of $Z^n$ to $Z$. As we will see in the proof of Theorem \ref{the-result}, the
study of $\|Z^n_{t_k}-Z_{t_k}\|$ goes through the study of
$\|Z^n_{t_k}-\hat{Z}^n_{t_k}\|$ and $\|\hat{Z}^n_{t_k}-Z_{t_k}\|$.\\

Before defining the discretized version of $(\nabla X_t)_t$ and $(D_s X_t)_{s
  \le t}$, we shortly introduce the  discretized Malliavin derivative and refer the reader 
to \cite{Bender-Parcz}
for more information on this topic.

\begin{defi}[Definition of $T_{_{m,+}}$, $T_{_{m,-}}$ and $\mathcal{D}^n_m$] \label{T-m}
For any function $F:\{-1,1 \}^n \to \rset$, the mappings $T_{_{m,+}}$ and
$T_{_{m,-}}$ are defined by
 \equa
T_{_{m, \pm}} F(\varepsilon_1, \dots, \varepsilon_n) := F(\varepsilon_1, \dots, \varepsilon_{m-1}, \pm 1, \varepsilon_{m+1}, \dots, \varepsilon_n), \quad \quad 1 \leq m \leq n.
\tion
For any $\xi= F(\varepsilon_1, \dots, \varepsilon_n)$, the discretized
Malliavin derivative is defined by
\equal \label{Dm}
 \mathcal{D}^n_m \xi :=\frac{ \e [\xi \eps_{m} |\sigma( ( \eps_l)_{l \in  \{1,...,n\} \setminus \{m\}} ) ]}{\sqrt{h}} = \frac{T_{_{m,+}} \xi - T_{_{m,-}} \xi}{2\sqrt{h}}, \quad 1 \leq m \leq n.
 \tionl
\end{defi}

\begin{defi}[Definition of $\phi_x^{(k,l)}$]\label{def:D-D-quotient}
Let $\phi$  be a $C^{0,1}([0,T]\times \R)$ function. We denote
\begin{align*}
\phi_x^{(k,l)} := \frac{\mathcal{D}^n_k \phi(t_l,X^n_{t_{l-1}})}{\mathcal{D}^n_k X^n_{t_{l-1}}} := \int_{0}^1 \phi_x(t_l, \vartheta T_{_{k,+}} \, X^n_{t_{l-1}} + (1-\vartheta) T_{_{k,-}} \, X^n_{t_{l-1}}) d\vartheta.
\end{align*}
If $\mathcal{D}^n_{k} X^n_{t_{\ell-1}}\neq 0$ the second $':='$ holds as an identity.
\end{defi}

We are now able to define the discretized version of $(\nabla X_t)_t$ and $(D_s X_t)_{s
  \le t}$.

\begin{defi}[Discretized processes $(\nabla X^{n,t_k,x}_{t_m})_{m \in
  \{k,\dots,n\}}$ and $(\mathcal{D}^n_k X^{n}_{t_m})_{m \in
  \{k,\dots,n\}}$]
For all $m$ in $\{k,\dots,n\}$ we define
\begin{align}
  \nabla X^{n,t_k,x}_{t_m} & = 1 + h\sum_{l=k+1}^m b_x(t_l,X^{n,t_k,x}_{t_{l-1}})\nabla X^{n,t_k,x}_{t_{l-1}}+ \sqrt{h}\!\! \sum_{l=k+1}^m
  \sigma_x(t_l,X^{n,t_k,x}_{t_{l-1}})\nabla X^{n,t_k,x}_{t_{l-1}}  \ep_l, \,\,
                              \quad  0\le k  \le n, \notag \\ 
	\mathcal{D}^n_k X^n_{t_m} & = \sigma(t_{k},X^n_{t_{k-1}}) +  h\sum_{l=k+1}^m  b_x^{(k,l)} 
	\mathcal{D}^n_k X^n_{t_{l-1}} + \sqrt{h} \sum_{l=k+1}^m \sigma_x^{(k,l)} (\mathcal{D}^n_k X^n_{t_{l-1}}) \ep_l, \quad  0<k \le n. \label{D^n_k-representation}
\end{align}
\end{defi}

\begin{rem}\begin{enumerate}[(i)]
\item Although $\nabla X^{n,t_k,X^n_{t_k}}_{t_m}$ is not equal to
$\tfrac{\mathcal{D}^n_{k+1}X^n_{t_m} }{\sigma(t_{k+1},X^n_{t_k})}$, we can
show that the difference of these terms converges in $L_p$ (see Lemma \ref{nabla and Malliavin}).
\item With the notation introduced above, \eqref{cZ-exactly-1} rewrites to
\equal  \label{cZ-exactly}
     Z^n_{t_k}     
 &=& \e^{\cg}_{k}  \left (  \mathcal{D}^n_{k+1} g(X^n_T) \right ) +\e^{\cg}_{k}\left (h\sum_{m=k+1}^{n-1}  
 \mathcal{D}^n_{k+1} f( t_{m+1},X^n_{t_m}, Y^n_{t_m} , Z^n_{t_m}) \right ).
 \tionl
 \end{enumerate}
\end{rem}

In order to define the discrete counterpart to \eqref{Z-representation-1}, we
first define the discrete counterpart to $(N^t_s)_{s \in [t,T]}$ given in
\eqref{N-cont}:
\equal \label{N-discrete}
  N^{n,t_k}_{t_{\ell}}:=  \sqrt{h} \sum_{m=k+1}^\ell \frac{\nabla X^{n,t_k,X^n_{t_k}}_{t_{m-1}}}{\sigma(t_m, X^n_{t_{m-1}})} \frac{\ep_m}{t_\ell - t_k}, \quad k < \ell \leq n.
\tionl
Notice that there is some constant $\widehat \kappa_2>0$ depending on $b,\sigma,T,\delta$ such that
\equal \label{norm-N}
 \left ( \e^{\cg}_{k} | N^{n,t_k}_{t_{\ell}} |^2\right )^\half \le \frac{ \widehat \kappa_2}{(t_{\ell}-t_k)^\half}, \quad 0\le k < \ell  \le n.
  \tionl

\begin{defi}[Discrete counterpart to \eqref{Z-representation}]
Let the process $\hat{Z}^n = (\hat{Z}^n_{t_k})_{k=0}^{n-1}$ be defined by 
\begin{align} \label{wrong-discreteZ}
  \hat{Z}^n_{t_k} := \e^{\cg}_{k}  \left (  \mathcal{D}^n_{k+1} g(X^n_T) \right ) 
  +   \e^{\cg}_{k} \left ( h\sum_{m= k+1}^{n-1}f( t_{m+1}, X^n_{t_m}, Y^n_{t_m}, Z^n_{t_m})   N^{n,t_k}_{t_m}\right ) \sigma(t_{k+1},X^n_{t_k}) ,
\end{align}
\end{defi}

\begin{rem}
 In \eqref{wrong-discreteZ} We could have used also the  approximate expression $\e^{\cg}_{k}  ({ g(X^n_T)    N^{n,t_k}_{t_n} \sigma(t_{k+1},X^n_{t_k}) }),$ but since we will  assume that  $g''$ exists, 
we  work with the  correct term. 
\end{rem}

The study of the convergence $\e^{\cg}_{0,x} | Z^n_{t_k} -  \hat{Z}^n_{t_k}
|^2$ requires stronger assumptions on the coefficients $b$,
$\sigma$, $f$ and $g$.

\begin{hypo}\label{Xgeneral}
   Assumptions  \ref{hypo2} and  \ref{hypo3} hold. 
 Additionally, we assume that all first and  second derivatives w.r.t.~the  variables $x,y,z$ of $b(t,x), \sigma(t,x)$ and $f(t,x,y,z)$  exist and are bounded Lipschitz functions w.r.t.~these  variables, uniformly in time.
   Moreover,  $g''$ satisfies \eqref{eq5}.
\end{hypo}

 \begin{prop}\label{discreteZand-wrongZdifference}
  If  Assumption \ref{Xgeneral} holds, then
      \equa
   \e^{\cg}_{0,x} | Z^n_{t_k} -  \hat{Z}^n_{t_k} |^2 \le  C_{\ref{discreteZand-wrongZdifference}}  \hat \Psi^2(x) h^{\alpha},
   \tion
   where $\e^{\cg}_{0,x}:=\e^{\cg}(\cdot|X_0=x)$, the function $\hat \Psi$ is defined in \eqref{def-hat-psi} below, and $C_{\ref{discreteZand-wrongZdifference}}$ depends on $b,\sigma, f, g,T,$ $p_0$ and $\delta$.
 \end{prop}	
 \begin{proof} 	
 According to  \cite[Proposition 5.1]{Briand-etal}  one has the  representations  
\equal \label{representations}
Y^n_{t_m} = u^n(t_m, X^n_{t_m}), \quad \text{ and} \quad  Z^n_{t_m} =\mathcal{D}^n_{m+1} u^n(t_{m+1},X^{n}_{t_{m+1}}), 
\tionl
where $u^n$ is the solution of the  finite difference equation \eqref{discrete-pde} with terminal condition $u^n(t_n,x)= g(x).$  Notice that by the definition of 
$\mathcal{D}^n_{m+1}$  in \eqref{Dm} the expression  $\mathcal{D}^n_{m+1} u^n(t_{m+1},X^{n}_{t_{m+1}})$ depends in fact on $X^{n}_{t_m}.$ 
Hence we can put 
\equa
f(t_{m+1}, X^n_{t_m}, Y^n_{t_m}, Z^n_{t_m}) 
&=&  f(t_{m+1}, X^n_{t_m}, u^n(t_m, X^n_{t_m}),  \mathcal{D}^n_{m+1} u^n(t_{m+1},X^{n}_{t_{m+1}}))\\
&=:& F^n(t_{m+1}, X^n_{t_{m}}).
\tion
From  \eqref{wrong-discreteZ} and  \eqref{cZ-exactly} we conclude that  (we use $\e:=\e^{\cg}_{0,x}$  for $\|\cdot \|$)
\equa
 && \les  \| Z^n_{t_k} -  \hat{Z}^n_{t_k} \| \\&=&\Bigg \|  \e^{\cg}_k \Bigg ( h\sum_{m=k+1}^{n-1}    \mathcal{D}^n_{k+1}   f(t_{m+1},X^n_{t_m}, Y^n_{t_m}, Z^n_{t_m})  \Bigg)  \\
&& \quad \quad  - \e^{\cg}_k  \left (  h\sum_{m=k+1}^{n-1}f( t_{m+1}, X^n_{t_m}, Y^n_{t_m}, Z^n_{t_m})   N^{n,t_k}_{t_m} \sigma(t_{k+1},X^n_{t_k}) \right)  \notag\Bigg \| \\
&\le&\!\!\!  \! \sum_{m=k+1}^{n-1} \frac{h}{m - k}  \sum_{\ell=k+1}^m \Bigg\|    \e^{\cg}_k \! \Bigg [ \mathcal{D}^n_{k+1}  F^n(t_{m+1}, X^n_{t_m})- \mathcal{D}^n_{ \ell } F^n(t_{m+1}, X^n_{t_m})  \frac{\sigma(t_{k+1},X^n_{t_k})\nabla X^{n,t_k,X^n_{t_k}}_{t_{\ell -1 }}}{\sigma(t_\ell, X^n_{t_{\ell - 1}})} \Bigg ]
 \Bigg \|.
\tion
With  the  notation introduced in Definition \ref{def:D-D-quotient} applied to $F^n$,
\equa
&& \less \bigg \|   \mathcal{D}^n_{k+1}F^n(t_{m+1}, X^n_{t_m})-\mathcal{D}^n_\ell F^n(t_{m+1}, X^n_{t_m}) \frac{\sigma(t_{k+1},X^n_{t_k})\nabla X^{n,t_k,X^n_{t_k}}_{t_{\ell-1}}}{\sigma(t_\ell,X^n_{t_{\ell-1}})} \bigg \|	\\  
&\le &  \|   (\mathcal{D}^n_{k+1}X^n_{t_m} )(  F^{n,(k+1,m+1)}_x -F^{n,(\ell,m+1)}_x) \| \\
&&+   \,\bigg \|  F^{n,(\ell,m+1)}_x    \bigg ( (\mathcal{D}^n_{k+1} X^n_{t_m})  -  
(\mathcal{D}^n_\ell X^n_{t_m})
 \frac{\sigma(t_{k+1},X^n_{t_k})\nabla X^{n,t_k,X^n_{t_k}}_{t_{\ell-1}}}{\sigma(t_\ell,X^n_{t_{\ell-1}})} \bigg) 
\bigg \| \\
&=:& A_1 +A_2.
\tion	
For $A_1$  we  use Definition \ref{def:D-D-quotient} again and  exploit the fact that 
$$x \mapsto F^n_x(t_{m+1},x):=\partial_xf(t_{m+1}, x, u^n(t_m, x),  \mathcal{D}^n_{m+1} u^n(t_{m+1},X^{n,t_m,x}_{t_{m+1}}))$$ 
is locally $\alpha$-H\"older continuous according to \eqref{D-Fx-difference}. By  Hölder's inequality and
Lemma \ref{nabla and Malliavin} \eqref{nM-i} and \eqref{nM-iii},
\begin{align*}
A_1 & \leq  \| \cD^n_{k+1} X^n_{t_m}\|_4\int_{0}^1 \| F^n_x(t_{m+1}, \vartheta T_{_{k+1,+}} X^n_{t_m} + (1-\vartheta) T_{_{k+1,-}}X^n_{t_m})  \\
&\quad \quad \quad \quad\quad \quad-   F^n_x(t_{m+1}, \vartheta T_{_{\ell,+}}  X^n_{t_m} + (1-\vartheta) T_{_{\ell,-}} X^n_{t_m})\|_4
 d\vartheta \leq  C(b,\sigma,f,g,T,p_0) \hat \Psi(x) h^{\frac{\alpha}{2}}.
\end{align*}
For the estimate of  $A_2$  we notice that  by our  assumptions the $L_4$-norm of $F^{n,(\ell,m+1)}_x  $  is bounded by $C  \Psi^2(x),$ so that  it suffices to estimate 
\equal \label{splitting-into-three-parts-D-k}
&& \bigg \| (\mathcal{D}^n_{k+1} X^n_{t_m})  -   (\mathcal{D}^n_\ell X^n_{t_m}) \frac{\sigma(t_{k+1},X^n_{t_k})\nabla X^{n,t_k,X^n_{t_k}}_{t_{\ell-1}}}{\sigma(t_\ell, X^n_{t_{\ell-1}})}\bigg \|_4  
\notag\\\!\!&&
\le  \left\| (\mathcal{D}^n_{k+1} X^n_{t_m}) - \frac{\sigma( t_{k+1},X^n_{t_k})  \, \mathcal{D}^n_\ell  X^n_{t_m}}{ \sigma(t_\ell ,X^n_{t_{\ell-1}}) } \frac{  \mathcal{D}^n_{k+1} 
X^n_{t_{\ell -1}}}{ \sigma( t_{k+1},X^n_{t_k})}\right\|_4  
 \notag\\&& \quad \quad \quad
 + \left\| \frac{\sigma( t_{k+1},X^n_{t_k}) \, \mathcal{D}^n_\ell X^n_{t_m}}{ \sigma(t_\ell,X^n_{t_{\ell-1}} )}
\Bigg (\nabla X^{n,t_k,X^n_{t_k}}_{t_{\ell-1}}  -    \frac{ \mathcal{D}^n_{k+1} X^n_{t_{\ell-1}}}{\sigma( t_{k+1},X^n_{t_k})} \Bigg )\right\|_4.
 \tionl
The second expression on the r.h.s.~of \eqref{splitting-into-three-parts-D-k}
is bounded by $C(b,\sigma,T,\delta)h^{\frac{1}{2}}$ as a consequence of
Lemma \ref{nabla and Malliavin} \eqref{nM-iv}-\eqref{nM-iii}. To show that also the first expression  is bounded by $C(b,\sigma,T,\delta)h^{\frac{1}{2}}$, we rewrite it using  \eqref{D^n_k-representation} 
and get 
\begin{align}\label{nabla-mvdifference-l2-error}
& \left |\frac{ \mathcal{D}^n_{\ell} X^n_{t_m}}{ \sigma(t_\ell,X^n_{t_{\ell-1}})}   \mathcal{D}^n_{k+1} X^n_{t_{\ell-1}} - \mathcal{D}^n_{k+1} X^n_{t_m}\right | \nonumber\\
& = \Bigg | \Bigg (1+ \sum_{l=\ell+1}^{m}  \frac{ \mathcal{D}^n_{\ell} X^n_{t_{l-1}}}{ \sigma(t_\ell,X^n_{t_{\ell-1}}) }  ( b_x^{(\ell,l)} h + \sigma_x^{(\ell,l)}  \sqrt{h} \varepsilon_l) \Bigg ) \nonumber \\
 & \quad \quad \times \Bigg (\sigma(t_{k+1},X^n_{t_k})+ \sum_{l=k+2}^{\ell-1}\mathcal{D}^n_{k+1} X^n_{t_{l-1}}( b^{(k+1,l)}_x h + \sigma^{(k+1,l)}_x \sqrt{h} \varepsilon_l)   \Bigg ) \nonumber\\
 & \quad \quad - \Bigg (\sigma(t_{k+1},X^n_{t_k})+ \bigg (\sum_{l=k+2}^{\ell-1} +  \sum_{l=\ell}^{m} \bigg )\mathcal{D}^n_{k+1} X^n_{t_{l-1}}( b^{(k+1,l)}_x h + \sigma^{(k+1,l)}_x \sqrt{h} \varepsilon_l) \Bigg ) \Bigg | \nonumber\\
&\le \big| \mathcal{D}^n_{k+1} X^n_{t_{\ell-1}}( b^{(k+1,\ell)}_x h + \sigma^{(k+1,\ell)}_x  \sqrt{h} \varepsilon_\ell)\big| \nonumber\\
&  \quad +   \Bigg | \sum_{l=\ell+1}^{m}  \bigg [\frac{ \mathcal{D}^n_{\ell} X^n_{t_{l-1}}}{ \sigma(t_\ell,X^n_{t_{\ell-1}}) }\mathcal{D}^n_{k+1} X^n_{t_{\ell-1}} - \mathcal{D}^n_{k+1} X^n_{t_{l-1}} \bigg] \Big( b^{(\ell,l)}_x h + \sigma^{(\ell,l)}_x \sqrt{h} \varepsilon_l\Big)  \Bigg | \nonumber\\
& \quad + \!  \Bigg | \sum_{l=\ell+1}^{m}  \mathcal{D}^n_{k+1} X^n_{t_{l-1}} \bigg[ b^{(\ell,l)}_x h + \sigma^{(\ell,l)}_x \sqrt{h} \varepsilon_l - \Big(b^{(k+1,l)}_x h + \sigma^{(k+1,l)}_x\sqrt{h} \varepsilon_l\Big) \bigg]  \Bigg |.
\end{align}
We take the $L_4$-norm of \eqref{nabla-mvdifference-l2-error} and apply the BDG inequality and  H\"older's inequality.
The second term on the r.h.s. of  \eqref{nabla-mvdifference-l2-error} will be  used for Gronwall's lemma, 
while the first and the last one can be bounded by $
  C(b,\sigma,T)h^{\frac{1}{2}},$  by using Lemma \ref{nabla and
    Malliavin}-\eqref{nM-iii}. For the last term we also use the
Lipschitz continuity of $b_x$ and $\sigma_x$ in space and  Lemma \ref{nabla and Malliavin}-\eqref{nM-i}.
\end{proof}	

\section{Main results}  \label{3}
{\ce In order to compute  the mean square distance between  the solution  to \eqref{BSDE} and  the solution to \eqref{discrete-BSDE} we construct
the random walk $B^n$ from the Brownian motion $B$ by Skorohod embedding.  Let 
\equal  \label{tau-k}
\tau_0 :=0  \quad \text{and} \quad   \tau_k := \inf\{ t> \tau_{k-1}: |B_t-B_{\tau_{k-1}}| = \sqrt{h} \}, \quad k \ge 1.
\tionl }
Then  $(B_{\tau_{k}}   -B_{\tau_{k-1}})_{k=1}^\infty$ is a sequence of i.i.d. random variables  with 
\[\PP(  B_{\tau_{k}}   -B_{\tau_{k-1}} = \pm \sqrt{h})= \tfrac{1}{2},
\]
which means that $\sqrt{h} \ep_k  \stackrel{d}{=}  B_{\tau_{k}}
-B_{\tau_{k-1}}\!.$ We will denote by $\e_{\tau_k}\!$ the conditional expectation
w.r.t.~$\cf_{\tau_k} \!\!:=\cg_k.$ In this case we also use the notation 
$\cX_{\tau_k}:=X^n_{t_k}$ for all $k =0, \dots,n,$ so that \eqref{Xn} turns into 
\begin{align*}
\cX_{\tau_k} = x + \sum_{j=1}^k b(t_j,\cX_{\tau_{j-1}})h + \sum_{j=1}^k \sigma(t_j,\cX_{\tau_{j-1}})(B_{\tau_j} - B_{\tau_{j-1}}), \quad 0 \leq k \leq n.
\end{align*}
{\ce \begin{hypo} \label{Bn-from-B}
We assume that the random walk  $B^n$  in \eqref{discrete-BSDE} is given by 
\equa 
B^n_{t} =  \sum_{k=1}^{[t/h] }(B_{\tau_{k}}   -B_{\tau_{k-1}}), \quad \quad 0 \leq t
\leq T,
\tion
where the $\tau_k, \, k =1,...,n$ are taken from  \eqref{tau-k}.
\end{hypo} }
\begin{rem}\label{B-n-rate}  Note that for $p>0$ there exists a $C(p) >0$ such that   for all $k =1, \dots,n$ it holds  $$ \tfrac{1}{C(p)}     (t_k h)^{\frac{1}{4}} \le    ( \e | B_{\tau_k} - B_{t_k}|^p )^{\frac{1}{p}} \le  C(p)     (t_k h)^{\frac{1}{4}}.$$ 
The upper estimate is given in Lemma \ref{B-difference}. For $p\in [4,\infty)$  the lower estimate follows from  \cite[Proposition 5.3]{ankirch}.
We get the lower estimate for $p\in (0,4)$ by choosing $0<\theta <1$   and $0<p<p_1$ such that   $ \frac{1}{4} = \frac{1-\theta}{p} + \frac{\theta}{ p_1}.$   
Then it holds by the log-convexity of $L_p$ norms (see, for example \cite[Lemma 1.11.5]{tao}) that 
$$\| B_{\tau_{k}} -B_{t_k}\|^{1-\theta}_p \ge  \frac{ \| B_{\tau_{k}}   -B_{t_k}\|_4}{   \| B_{\tau_{k}}   -B_{t_k}\|^{\theta}_{p_1}} 
   \ge  \frac{C(4)^{-1}   (t_k h)^{\frac{1}{4}} }{ \Big(C(p_1) (t_k h)^{\frac{1}{4}} \Big)^{\theta}  } \ge  \Big (C(p)  (t_k h)^{\frac{1}{4}} \Big )^{1-\theta}.  $$
Since for  $t\in [t_k,t_{k+1})$ it holds $B^n_t=  B_{\tau_k}$ and $\|B_t -B_{t_k} \|_p  \le C(p) h^\frac{1}{2},$  we have for any $p>0$   that    
\equal \label{B-n-minus-B-rate}
\sup_{0 \le t\le T}\|B^n_t -B_t \|_p = O(h^\frac{1}{4}).
\tionl
\end{rem} 

  Proposition \ref{no f} states the convergence rate of $(Y_v,Z_v)$ to
$(Y^n_v,Z^n_v)$ in $L_2$ when $f=0$ and Theorem \ref{the-result} generalizes
this result  for  any $f$ which satisfies Assumption \ref{Xgeneral}.
  \begin{prop} \label{no f}  Let {\ce Assumptions  \ref{hypo2} and \ref{Bn-from-B} 
  hold.} If  $f =0$ and  $g \in C^1$   is such that $g'$ is a locally $\alpha$-H\"older continuous function in the sense of \eqref{eq5},
then  for all $0 \le v < T$, we have (for sufficiently large  $n$) that
\equa
 \e_{0,x} |Y_v - Y^n_v|^2     \le  C^y_{\ref{no f}}  \Psi(x)^2  h^{\frac{1}{2}}, \quad  \text{ and }  \quad  \e_{0,x} |Z_v - Z^n_v|^2 \le   C^z_{\ref{no f}} \Psi(x)^2  h^\frac{\alpha}{2},  \tion
where   $C^y_{\ref{no f}}={\ce C(C_g, b, \sigma, T, p_0, \delta)}$ and $C^z_{\ref{no f}}= C(C_{g'}, b, \sigma, T, p_0, \delta).$
\end{prop}

\begin{thm} \label{the-result} Let   {\ce Assumptions \ref{Xgeneral}  and   \ref{Bn-from-B}  be satisfied.}
Then  for all $v \in [0,T)$ and large enough $n$, we have
\begin{align*}
  \e_{0,x} |Y_v - Y^n_v|^2   +     \e_{0,x} |Z_v - Z^n_v|^2      \le   C_{\ref{the-result}}  \hat{\Psi}(x)^2h^{ \frac{1}{2} \wedge \alpha}  
\end{align*}
with  $C_{\ref{the-result}}={\ce C(b,\sigma,f,g,T,p_0,\delta)}$ and $\hat{\Psi}$ is given in \eqref{def-hat-psi}.
\end{thm}

 {\ce
  \begin{rem}
 As noticed above, the filtration $\mathcal{G}_k$ coincides with
$\mathcal{F}_{\tau_k}$, for all $k = 0,\dots,n$. The expectation $\e_{0,x}$ appearing
in Proposition \ref{no f} and in Theorem \ref{the-result} is defined on the
probability space $(\Omega,\mathcal{F},\mathbb{P})$.
\end{rem}}
{\ce \begin{rem} In order to avoid too much notation for the  dependencies of the constants, if for example  only $g$ is mentioned  and not $C_g,$   this means that the estimate might depend also on the bounds of the derivatives of $g.$
\end{rem}}

 From  \eqref{B-n-minus-B-rate} one can see that   the convergence rates stated in  Proposition \ref{no f} and Theorem \ref{the-result}  
are the  natural ones for this approach.  
The results are proved in the next two
sections. In both proofs, we will use the following remark.

\begin{rem}
  Since the process $(X_t)_{t \ge 0}$ is strong Markov we can express conditional expectations with the help of an independent copy of $B$ denoted by  $\ti B,$ for example  
 $\e_{\tau_k} g(X^n_T) = \ti \e g(   \ti \cX^{\tau_k,\cX_{\tau_k}}_{\tau_n})$  for $0 \leq k \leq n$,
where 
\equal  \label{X-n-with-tilde}
\ti \cX^{\tau_k,\cX_{\tau_k}}_{\tau_n} = \cX_{\tau_k}+ \sum_{j=k+1}^n b(t_j, \ti \cX^{\tau_k,\cX_{\tau_k}}_{\tau_{j-1}})h + \sum_{j=k+1}^n 
\sigma(t_j, \ti \cX^{\tau_k,\cX_{\tau_k}}_{\tau_{j-1}})(\ti B_{\ti \tau_{j-k}} - \ti B_{\ti \tau_{j-k-1}}), 
\tionl
(we define  $\ti \tau_k:= 0$ and    $\ti \tau_j:=   \inf\{ t> \ti \tau_{j-1}: |\ti B_t- \ti B_{\ti \tau_{j-1}}| = \sqrt{h} \}$ for $j \ge 1$ and $\tau_n := \tau_k +\ti \tau_{n-k}$
for $n\ge k$). 
 In fact, to represent the conditional expectations   $\e_{t_k}$ and $\e_{\tau_k}$  we work here with $\ti \e$  and the Brownian motions $B'$ and $B'',$ respectively, given by 
\equal \label{independent-B}
 B'_t = B_{t\wedge t_k} + \tilde B_{(t-t_k)^+} \quad \text{ and } \quad  B''_t = B_{t\wedge \tau_k} + \tilde B_{(t-\tau_k)^+}, \quad t\ge 0.  
 \tionl
\end{rem}

\subsection{Proof of Proposition \ref{no f}: the approximation rates for the  zero generator case } 
 To shorten the notation, we use $\e :=\e_{0,x}.$
 Let us first deal with the error of
  $Y$. If $v$ belongs to $[t_k,t_{k+1})$ we have $Y^n_v=
  Y^n_{t_k}$.
  Then
  \begin{align*}
   \e |Y_v - Y^n_v|^2 \le 2( \e |Y_v - Y_{t_k}|^2 +  \e |Y_{t_k} - Y^n_{t_k}|^2).
\end{align*}
Using  Theorem \ref{difference-estimates for Y and Z} we bound $\|Y_v - Y_{t_k}\|$ by 
$$C^y_{\ref{difference-estimates for Y and Z}}  \Psi(x) (v-t_k)^\frac{1}{2} = {\ce  C(C_g, b, \sigma, T,p_0, \delta)}  \Psi(x) (v-t_k)^\frac{1}{2} $$  
(since $\alpha=1$ can be chosen when $g$ is locally Lipschitz continuous). It remains to bound 
\equa
\e |Y_{t_k} -Y^n_{t_k}|^2&=&\e|\e_{t_k} g(X_T) -\e_{\tau_k} g(X^n_T) |^2   
=\e | \ti \e g(\ti X^{t_k,X_{t_k}}_{t_n})-\ti \e g(\ti \cX^{\tau_k,\cX_{\tau_k}}_{\tau_n})|^2.
\tion
By \eqref{eq5} and the Cauchy-Schwarz inequality ($\Psi_1:=C_g(1+|\ti
X^{t_k,X_{t_k}}_{t_n}|^{p_0}+|  \tilde \cX^{\tau_k,{\cX}_{\tau_k}}_{\tau_n}|^{p_0})$),
\equa
 | \ti \e g(\ti X^{t_k,X_{t_k}}_{t_n})-\ti \e g(\ti \cX^{\tau_k,\cX_{\tau_k}}_{ \tau_n})|^2  
 &\le& (\ti \e (\Psi_1 |\ti X^{t_k,X_{t_k}}_{t_n}-\ti \cX^{\tau_k,\cX_{\tau_k}}_{\ \tau_n}| ))^2 
 \le  \ti \e (\Psi_1^2) \ti \e |\ti X^{t_k,X_{t_k}}_{t_n}-\ti \cX^{\tau_k,\cX_{\tau_k}}_{ \tau_n}|^{2}.
 \tion
Finally, we get by Lemma \ref{X-Lemma}-\eqref{XandcX} that 
\equa
\e |Y_{t_k} -Y^n_{t_k}|^2&\le&  \left( \e\ti \e (\Psi_1^4)  \right)^{\frac{1}{2}}\left(\e\ti \e|\ti X^{t_k,X_{t_k}}_{t_n}-\ti \cX^{\tau_k,\cX_{\tau_k}}_{\tau_n}|^{4}\right)^{\frac{1}{2}}
\le  C(C_g, b, \sigma, T,p_0) \Psi(x)^2 h^\frac{1}{2}.
 \tion 
\smallskip

Let us now deal with the error of $Z$. We use $  \|Z_v - Z^n_v \| \le   \|Z_v - Z_{t_k}\| +   \|Z_{t_k} - Z^n_{t_k}\|$   
and  the representation 
$$ Z_{t} = \sigma(t,X_{t})  \ti \e( g'(\ti X^{t,X_{t}}_{T}) \nabla \ti X^{t,X_{t}}_{T})  $$ 
(see Theorem \ref{thm1}), where 
 \equal 
  \ti X^{t,x}_s &=& x+\int_t^s b(r, \ti X^{t,x}_r)  dr + \int_t^s \sigma(r,\ti X^{t,x}_r) d\ti B_{r-t},  \label{tilde-X} \\
  \nabla \ti X^{t,x}_s &=& 1+\int_t^s b_x(r, \ti X^{t,x}_r) \nabla \ti X^{t,x}_r dr + \int_t^s \sigma_x(r,\ti X^{t,x}_r) \nabla \ti X^{t,x}_r d\ti B_{r-t}, \quad 0\le t\le s\le T.
 \notag \tionl 
For the  first term we get by the assumption on $g$ and Lemma
\ref{X-Lemma}-\eqref{Xsx-Xty} and \eqref{nablaX-diff}  
\equa
   \|Z_v - Z_{t_k}\| &=&  \|  \sigma(v,X_v)  \ti \e (g'(\ti X^{v,X_v}_{T}) \nabla \ti X^{v,X_v}_{T})  - \sigma(t_k,X_{t_k})  \ti \e( g'(\ti X^{t_k,X_{t_k}}_{T}) \nabla \ti X^{t_k,X_{t_k}}_{T}) \|\\
&\le& \|  \sigma(v,X_v) - \sigma(t_k,X_{t_k}) \|_4 \|  \ti \e (g'(\ti X^{v,X_v}_{T}) \nabla \ti X^{v,X_v}_{T}) \|_4 \\
&& + \|\sigma\|_\infty \|   \ti \e (g'(\ti X^{v,X_v}_{T}) \nabla \ti X^{v,X_v}_{T})  -  \ti \e( g'(\ti X^{t_k,X_{t_k}}_{T}) \nabla \ti X^{v,X_v}_{T})   \| \\
&& + \|\sigma\|_\infty \|   \ti \e( g'(\ti X^{t_k,X_{t_k}}_{T})  \nabla \ti X^{v,X_v}_{T})  -  \ti \e( g'(\ti X^{t_k,X_{t_k}}_{T}) \nabla \ti X^{t_k,X_{t_k}}_{T})  \| \\
&\le &  C( C_{g'},b,\sigma,T,p_0)\Psi(x) \Big [h^\half  +  \|X_v -X_{t_k} \|_4
+  \left ( \e\ti \e |\ti X^{v,X_v}_{T}-\ti X^{t_k,X_{t_k}}_{T}|^{4\alpha} 
  \right)^\frac{1}{4}\\
  &&+\left (\e\ti \e | \nabla \ti X^{v,X_v}_{T} -\nabla \ti X^{t_k,X_{t_k}}_{T}|^4 \right)^\frac{1}{4} \Big] \\
&\le&  C( C_{g'},b,\sigma,T,p_0) \Psi(x) h^\frac{\alpha}{2}.  
\tion
We compute the  second term  using  $Z^n_{t_k} $  as  given in
\eqref{cZ-exactly}. Hence, with the notation from Definition \ref{def:D-D-quotient},
\equa
  \|Z_{t_k} - Z^n_{t_k}\|^2  
&=& \e \big |  \sigma(t_k,X_{t_k})  \ti \e g'(\ti X^{t_k,X_{t_k}}_{t_n})   \nabla \ti X^{t_k,X_{t_k}}_{t_n}   
          -   \ti \e  \mathcal{D}^n_{k+1} g(\ti \cX^{\tau_k,\cX_{\tau_k}}_{\tau_n}) \big  |^2 \\
&\le &  \| \sigma\|_\infty ^2 \,   \e \left |  \ti  \e ( g'(\ti X^{t_k,X_{t_k}}_{t_n})   \nabla \ti X^{t_k,X_{t_k}}_{t_n} )  
          -  \frac{ \ti \e  \mathcal{D}^n_{k+1} g(\ti \cX^{\tau_k,\cX_{\tau_k}}_{\tau_n})}{   \sigma(t_k,X_{t_k})  } \right |^2 \\            
&=& \| \sigma\|_\infty ^2 \,  \e \Big |   \ti \e (g'(\ti X^{t_k,X_{t_k}}_{t_n}) \nabla \ti X^{t_k,X_{t_k}}_{t_n} )
 -      \ti \e  \Big ( g_x^{(k+1,n+1)}
      \frac{  \mathcal{D}^n_{k+1} \ti \cX^{\tau_k,\cX_{\tau_k}}_{\tau_n}}{   \sigma(t_k,X_{t_k})  }\Big ) \Big |^2. \\
    \tion
    We insert  $\pm  \ti \e ( g_x^{(k+1,n+1)} \nabla \ti X^{t_k,X_{t_k}}_{t_n})$ and get  by the Cauchy-Schwarz inequality that
\equal \label{split}
 &&  \les \Big |   \ti \e (g'(\ti X^{t_k,X_{t_k}}_{t_n}) \nabla \ti X^{t_k,X_{t_k}}_{t_n}) 
 -     \ti \e  \Big ( g_x^{(k+1,n+1)}
      \frac{  \mathcal{D}^n_{k+1} \ti \cX^{\tau_k,\cX_{\tau_k}}_{\tau_n}}{   \sigma(t_k,X_{t_k})  } \Big )  \Big |^2 \notag\\
&&\le 2  \ti \e | g'(\ti X^{t_k,X_{t_k}}_{t_n})- g_x^{(k+1,n+1)}|^2    \ti \e  | \nabla \ti X^{t_k,X_{t_k}}_{t_n} |^2 
         + 2\ti \e  |  g_x^{(k+1,n+1)}|^2  \ti \e  \Big |  \nabla \ti X^{t_k,X_{t_k}}_{t_n} - \frac{  \mathcal{D}^n_{k+1} \ti \cX^{\tau_k,\cX_{\tau_k}}_{\tau_n}}{   \sigma(t_k,X_{t_k})  }  \Big|^2. \notag\\
         \tionl
For the estimate of $ \ti \e  | \nabla \ti X^{t_k,X_{t_k}}_{t_n} |^2$ we use   Lemma   \ref{X-Lemma}. Since $g'$  satisfies  \eqref{eq5} we proceed with         
 \equa  
&& \less  \ti \e | g'(\ti X^{t_k,X_{t_k}}_{t_n})- g_x^{(k+1,n+1)}|^2   \\
&\le &  \int_0^1 \ti  \e  \Big | g'(\ti X^{t_k,X_{t_k}}_{t_n})
- g'(\vartheta T_{_{k+1,+}} \ti \cX^{\tau_k,\cX_{\tau_k}}_{\tau_n}+ (1-\vartheta) T_{_{k+1,-}} 
\ti \cX^{\tau_k,\cX_{\tau_k}}_{\tau_n}) \Big |^2 d \vartheta \\
&\le&  \int_0^1    (\ti \e \Psi_1^4)^\half \Big [ \ti \e  \left |\ti X^{t_k,X_{t_k}}_{t_n}
- \vartheta T_{_{k+1,+}} \ti \cX^{\tau_k,\cX_{\tau_k}}_{\tau_n} - (1-\vartheta) T_{_{k+1,-}} 
\ti \cX^{\tau_k,\cX_{\tau_k}}_{\tau_n}  \right|^{4\alpha} \Big ]^\half  d \vartheta, 
\tion      
where $\Psi_1:=C_{g'}(1+|\ti X^{t_k,X_{t_k}}_{t_n}|^{p_0}+|\vartheta T_{_{k+1,+}} \ti \cX^{\tau_k,\cX_{\tau_k}}_{\tau_n}+ (1-\vartheta) T_{_{k+1,-}} 
\ti \cX^{\tau_k,\cX_{\tau_k}}_{\tau_n}|^{p_0}).   $ 
For  $\ti \e \Psi_1^4$ and
\equa
&& \less \ti \e  \left |\ti X^{t_k,X_{t_k}}_{t_n}
- (\vartheta T_{_{k+1,+}} \ti \cX^{\tau_k,\cX_{\tau_k}}_{\tau_n}+ (1-\vartheta) T_{_{k+1,-}} 
\ti \cX^{\tau_k,\cX_{\tau_k}}_{\tau_n} )  \right|^{ 4\alpha} \\
&\le& 8  \left ( \vartheta^{2\alpha}  \ti \e  \left |\ti X^{t_k,X_{t_k}}_{t_n}
-  T_{_{k+1,+}} \ti \cX^{\tau_k,\cX_{\tau_k}}_{\tau_n}  \right|^{4\alpha}  + (1-\vartheta)^{2\alpha}  \ti \e  \left |\ti X^{t_k,X_{t_k}}_{t_n}
 - T_{_{k+1,-}} \ti \cX^{\tau_k,\cX_{\tau_k}}_{\tau_n}   \right|^{4\alpha} \right ) \\
 &\le & C(b,\sigma,T) h^{2 \alpha} + C(b,\sigma,T)( |X_{t_k} - \cX_{\tau_k}|^{4 \alpha}+  h^{\alpha}),
\tion
we use  Lemma \ref{nabla and Malliavin}   and  Lemma \ref{X-Lemma}-\eqref{XandcX}.
For the last term in \eqref{split} we notice that 
$$\e \ti \e  |  g_x^{(k+1,n+1)}|^4 \le  C(C_{g'}, b,\sigma,T,p_0)  \Psi^4(x).$$
By Lemma   \ref{X-Lemma} we have $\e\ti \e | \nabla  \ti
X^{t_k,X_{t_k}}_{t_n}- \nabla \ti \cX^{\tau_k,\cX_{\tau_k}}_{\tau_n}|^p    \le
 C(b,\sigma,T,p) h^{\frac{p}{4}},$
and by Lemma \ref{nabla and Malliavin},
\equa
&& \les \e\ti \e  \left |  \nabla \ti \cX^{\tau_k,\cX_{\tau_k}}_{\tau_n} 
-\frac{\mathcal{D}^n_{k+1} \ti \cX^{\tau_k,\cX_{\tau_k}}_{\tau_n}}{  \sigma(t_{k}, X_{t_k})}\right |^p   \\
&\le&\!\! C(p)  \e\left | \nabla X^{n,t_k, X^n_{t_k}}_{t_n}  - \dfrac{\mathcal{D}^n_{k+1}X^n_{t_n} }{\sigma(t_{k+1},X^n_{t_k})} \right |^p 
   + C(p) \e\left | \dfrac{\mathcal{D}^n_{k+1}X^n_{t_n} }{\sigma(t_{k+1},X^n_{t_k})} - \dfrac{\mathcal{D}^n_{k+1}X^n_{t_n} }{\sigma(t_k, X_{t_k})} \right |^p  
\le  C(b,\sigma,T,p,  \delta) h^{\frac{p}{4}}.
\tion 
Consequently, 
$
 \|Z_{t_k} - Z^n_{t_k}\|^2 \le  C(C_{g'}, b,\sigma,T,p_0,\delta) \Psi^2(x) h^\frac{\alpha}{2}.$

\subsection{Proof of Theorem \ref{the-result}: the approximation rates  for the general case } 
Let $u : [0,T)\times \R \to \R$ be the solution of the  PDE \eqref{pde} associated to \eqref{BSDE}. We use the representations
$ Y_s = u(s, X_s)$ and  $Z_s = \sigma(s, X_s)u_x(s, X_s)$ stated in Theorem \ref{thm1} and define
\equal  \label{theF}
F(s,x) := f(s, x, u(s, x),  \sigma(s, x)u_x(s, x)).
\tionl
From \eqref{BSDE} and  \eqref{discrete-BSDE} we conclude
\equa
 \|Y_{t_k}-  Y^n_{t_k} \| &\le  &  \| \e_{t_k} g(X_T)  -  \e_{\tau_k} g(X^n_T) \|  \notag
 \\&&  +
 \left \| \e_{t_k}\int_{t_k}^T f(s,X_s,Y_s,Z_s)ds  - h \e_{\tau_k}\sum_{m=k}^{n-1} f(t_{m+1},X^n_{t_m}, Y^n_{t_m},Z^n_{t_m})   \right \|,
 \tion
 where Proposition \ref{no f}  provides the estimate for the terminal
 condition. We decompose the generator term as follows:
 \equa
&&\les \e_{t_k} f(s,X_s,Y_s,Z_s)-\e_{\tau_k}f(t_{m+1}, X^n_{t_m},Y^n_{t_m},Z^n_{t_m})\\ &=&
[\e_{t_k} f(s,X_s,Y_s,Z_s)- \e_{t_k}f(t_m,X_{t_m},Y_{t_m},Z_{t_m})] + [\e_{t_k}F(t_m,X_{t_m}) -\e_{\tau_k}F(t_m, X^n_{t_m})]\\
&& + [\e_{\tau_k}F(t_m,X^n_{t_m}) -\e_{\tau_k} F(t_m,X_{t_m})]  +[\e_{\tau_k} f(t_m, X_{t_m},Y_{t_m},Z_{t_m}) - \e_{\tau_k}f(t_{m+1}, X^n_{t_m},  Y^n_{t_m},Z^n_{t_m})] \\
&=:& d_1(s,m) +d_2(m)+ d_3(m)+ d_4(m).
\tion
We use 
\equa
&& \pull \left \|    \e_{t_k}\int_{t_k}^T f(s,X_s,Y_s,Z_s)ds -   h \e_{\tau_k}\sum_{m=k}^{n-1} f(t_{m+1}, X^n_{t_m},Y^n_{t_m},Z^n_{t_m}) \right \| 
 \notag\\&& 
\quad\quad  \quad\quad\quad\quad\quad\quad\le \sum_{m=k}^{n-1} \left (  \left \|  \int_{t_m}^{t_{m+1}}  d_1(s,m)  ds\right \|  +  h \sum_{i=2}^4  \|d_i(m) \| \right)
\tion
and estimate the expressions on the right hand side. 
For the function $F$ defined in \eqref{theF} we  use Assumption \ref{Xgeneral}  (which implies that \eqref{eq5} holds for $\alpha=1$) to derive by Theorem \ref{thm1} 
and the mean value theorem  that for $x_1, x_2 \in \R$  there exist 
 $\xi \in [\min\{x_1,x_2\},\max\{x_1,x_2\}] $ such that
\equal \label{F Lipschitz-new}
|F(t,x_1) -F(t,x_2)| &=& |f(t, x_1, u(t, x_1), \sigma(t,x_1)u_x(t, x_1)) - f(t, x_2, u(t, x_2), \sigma(t,x_2)u_x(t, x_2))| \notag \\
& \le &  C(L_f, \sigma)  \left  ( 1 +    c^2_{\ref{thm1}} \Psi( \xi )  +    \frac{c^3_{\ref{thm1}} \Psi(\xi)}{(T-t)^{\frac{1}{2}}} \right )|x_1-x_2|  \notag\\
& \le &  C(L_f, c^{2,3}_{\ref{thm1}}, \sigma, T) (1+|x_1|^{p_0+1}+|x_2|^{p_0+1}) \frac{|x_1-x_2|}{(T-t)^{\frac{1}{2}}}.
\tionl 
By   \eqref{Lipschitz}, standard estimates on $(X_s),$ Theorem \ref{difference-estimates for Y and Z}-$\eqref{one}$  and Proposition \ref{betterZ}   for $p=2$ we immediately get  
\equa
   \|    d_1(s,m)  \| 
&\le& C(L_f, C^y_{\ref{difference-estimates for Y and Z}},C_{\ref{betterZ}}, b,\sigma,T)  \Psi(x)  \,h^\half \\
&=& {\ce C(b,\sigma, f,g, T,p_0,\delta)}  \Psi(x)  \,h^\half.
\tion
For the estimate of $d_2$ one exploits  
$$\e_{t_k}F(t_m,X_{t_m}) -\e_{\tau_k}F(t_m,X^n_{t_m})= \ti \e F(t_m, \ti X^{t_k,X_{t_k}}_{t_m}) -   \ti \e F(t_m,\ti X^{n, t_k, X^n_{t_k}} _{t_m})   $$
and then uses  \eqref{F Lipschitz-new} and Lemma \ref{X-Lemma}-\eqref{XandcX}. This gives 
\equa
 \|d_2(m) \| &\le&  C(L_f, c^{2,3}_{\ref{thm1}}, b,\sigma,T, p_0)    \Psi(x)    \frac{1}{(T-t_m)^{\frac{1}{2}} } h^{\frac{1}{4}}.
\tion
 For $d_3$ we start with  Jensen's inequality and continue then similarly  as above to get 
\equa
 \|d_3(m) \| \le  \|  F(t_m, X^n_{t_m}) -  F(t_m,X_{t_m}) \| 
\le  C(L_f, c^{2,3}_{\ref{thm1}}, b,\sigma,T, p_0) \Psi(x)   \frac{1}{(T-t_m)^{\frac{1}{2}}}  h^\frac{1}{4},
\tion
and for the last term we get
\equa
 \|d_4(m) \| &\le & 
 L_f  (  h^\half + \| X_{t_m} - X^n_{t_m}\|    +  \| Y_{t_m} - Y^n_{t_m}\|+  \| Z_{t_m} -Z^n_{t_m}\| ).
\tion
This implies 
\equal \label{newYdiff} 
  \| Y_{t_k} - Y^n_{t_k}\|  \le  C  \Psi(x)  h^{\frac{1}{4}}+ h L_f \sum_{m=k}^{n-1}(  \| Y_{t_m} - Y^n_{t_m}\| +  \| Z_{t_m} -Z^n_{t_m}\|), 
\tionl
 where $C=  C(L_f, C^y_{\ref{no f}}, C^y_{\ref{difference-estimates for Y and
        Z}},  C_{\ref{betterZ}}, c^{2,3}_{\ref{thm1}},b,\sigma, T,p_0) = {\ce C(b,\sigma, f,g, T,p_0,\delta)}.$  \bigskip
        
For $\| Z_{t_k} -Z^n_{t_k}\|$ we use the representations  \eqref{Z-representation}, \eqref{cZ-exactly} and the approximation  \eqref{wrong-discreteZ} as well as  Proposition \ref{discreteZand-wrongZdifference}. Instead of 
$N^{n,t_k}_{t_n}$ we will use  here the notation $N^{n,\tau_k}_{\tau_n}$ to indicate its measurability w.r.t.~the filtration $(\cf_t).$   It holds that
\equal \label{Zand-its-approximations}
 \|Z^n_{t_k}  - Z_{t_k}\| &\le & \|  Z^n_{t_k} - \hat Z^n_{t_k}\| + \|Z_{t_k}- \hat Z^n_{t_k}\|  \notag\\
&\le &  C_{\ref{discreteZand-wrongZdifference}}  \hat \Psi(x) h^{\frac{\alpha}{2}} +   \|  \sigma(t_k,X_{t_k})  \ti \e g'(\ti X^{t_k,X_{t_k}}_{t_n})   \nabla \ti X^{t_k,X_{t_k}}_{t_n}   
          -   \ti \e  \mathcal{D}^n_{k+1} g(\ti X^{n,t_k,X^n_{t_k}}_{t_n}) \|   \notag\\
      &&+  \bigg  \|\e_{t_{k}}\int_{t_{k+1}}^T f(s,X_s,Y_s,Z_s) N^{t_k}_s ds \, \sigma(t_k,X_{t_k}) \notag \\
&& \quad \quad   -\e_{\tau_k} h\sum_{m=k+1}^{n-1}f( t_{m+1}, X^n_{t_m}, Y^n_{t_m},Z^n_{t_m})   N^{n,\tau_k}_{\tau_m} \sigma(t_{k+1},X^n_{t_k})
  \bigg \|    \notag \\
&& +     \bigg  \|\e_{t_k}\int_{t_k}^{t_{k+1}} f(s,X_s,Y_s,Z_s) N^{t_k}_s ds \, \sigma(t_k,X_{t_k})  \bigg \| .
\tionl
For the terminal condition  Proposition \ref{no f}  provides  
\equal \label{Z-terminal-cond} 
  \|  \sigma(t_k,X_{t_k})  \ti \e g'(\ti X^{t_k,X_{t_k}}_{t_n})   \nabla \ti X^{t_k,X_{t_k}}_{t_n}   
          -   \ti \e  \mathcal{D}^n_{k+1} g(\ti X^{n,t_k,X^n_{t_k}}_{t_n}) \|
          \le  (C^z_{\ref{no f}})^\half \Psi(x)  h^\frac{1}{4}. 
\tionl          
We continue with the generator terms and use $F$ defined in \eqref{theF}  to decompose the difference
\equa
&& \pull \pull \e_{t_k}f(s,X_s,Y_s,Z_s) N^{t_k}_s \sigma(t_k,X_{t_k})-\e_{\tau_k}f( t_{m+1}, X^n_{t_m}, Y^n_{t_m},Z^n_{t_m})  N^{n,\tau_k}_{\tau_m} \sigma(t_{k+1},X^n_{t_k})    \\
& = & \e_{t_k}f(s,X_s,Y_s,Z_s) N^{t_k}_s \sigma(t_k,X_{t_k})  -   \e_{t_k}f(t_m, X_{t_m},Y_{t_m},Z_{t_m}) N^{t_k}_{t_m}\sigma(t_k,X_{t_k})     \\
&& + \e_{t_k}F(t_m,X_{t_m}) N^{t_k}_{t_m} \sigma(t_k,X_{t_k})  -   \e_{\tau_k} F(t_m,X^n_{t_m}) N^{n,\tau_k}_{\tau_m} \sigma(t_{k+1},X^n_{t_k})      \\
&& + \e_{\tau_k}\left [ [  F(t_m,X^n_{t_m})   - F(t_m,X_{t_m})] N^{n,\tau_k}_{\tau_m} \sigma(t_{k+1},X^n_{t_k})    \right ] \\
&& + \e_{\tau_k} \left [ [f(t_m, X_{t_m},Y_{t_m},Z_{t_m})- f(t_{m+1}, X^n_{t_m}, Y^n_{t_m},Z^n_{t_m})]N^{n,\tau_k}_{\tau_m} \sigma(t_{k+1},X^n_{t_k})    \right ]  \\
&=:&  {\tt t}_1(s,m)+{\tt t}_2(m)+ {\tt t}_3(m) + {\tt t}_4(m)
\tion
where $s \in [t_m, t_{m+1}).$
For ${\tt t}_1$ we use that $\e_{t_k}f(t_m, X_{t_k},Y_{t_k},Z_{t_k}) (N^{t_k}_s  -N^{t_k}_{t_m}) =0,$ so that  
\equa
 \|{\tt t}_1(s,m)\| 
 &\le &   \| \e_{t_k}f(s,X_s,Y_s,Z_s) N^{t_k}_s \sigma(t_k,X_{t_k})  -   \e_{t_k}f(t_m, X_{t_m},Y_{t_m},Z_{t_m}) N^{t_k}_s\sigma(t_k,X_{t_k})   \| \\
 && +  \|  \e_{t_k}(f(t_m, X_{t_m},Y_{t_m},Z_{t_m})  - f(t_m, X_{t_k},Y_{t_k},Z_{t_k}) ) (N^{t_k}_s  -N^{t_k}_{t_m})\sigma(t_k,X_{t_k}) \|.
\tion
As before, we rewrite the conditional expectations with the help of  the independent copy $\ti B.$  Then 
\equa
&& \less \e_{t_k}f(s,X_s,Y_s,Z_s) N^{t_k}_s   -   \e_{t_k}f(t_m, X_{t_m},Y_{t_m},Z_{t_m}) N^{t_k}_s  \\
&=&   \ti \e [(f(s, \ti X^{t_k,X_{t_k}}_s, \ti Y^{t_k,X_{t_k}}_s, \ti Z^{t_k,X_{t_k}}_s) - f(t_m, \ti X^{t_k,X_{t_k}}_{t_m}, \ti Y^{t_k,X_{t_k}}_{t_m},\ti  Z^{t_k,X_{t_k}}_{t_m})) \ti N^{t_k}_s]
\tion
and
\equa
&& \less \e_{t_k}(f(t_m, X_{t_m},Y_{t_m},Z_{t_m})  - f(t_m, X_{t_k},Y_{t_k},Z_{t_k}) ) (N^{t_k}_s  -N^{t_k}_{t_m}) \\
&=& \ti \e [(f(t_m, \ti X^{t_k,X_{t_k}}_{t_m}, \ti Y^{t_k,X_{t_k}}_{t_m},\ti  Z^{t_k,X_{t_k}}_{t_m})   - f(t_m, X_{t_k},Y_{t_k},Z_{t_k}) ) (\ti N^{t_k}_s  - \ti N^{t_k}_{t_m})].                        
\tion
We apply  the  conditional H\"older inequality, and from   the estimates \eqref{norm-weight}   and  
$\ti \e |\ti  N^{t_k}_s   -  \ti N^{t_k}_{t_m}|^2 \le  C(b,\sigma,T,\delta)  \frac{h}{(s-t_k)^2}$  we get
\equa
 \|{\tt t}_1(s,m)\| 
 &\le & \frac{\kappa_2 \| \sigma \|_{\infty}}{(s-t_k)^\half}  \| f(s, X_s, Y_s,Z_s)- f(t_m, X_{t_m},Y_{t_m},Z_{t_m})\|  \\
 && + C(b,\sigma,T,\delta)  \frac{ h^{\half}}{s-t_k}\ \| f(t_m, X_{t_m},Y_{t_m},Z_{t_m})- f(t_k, X_{t_k},Y_{t_k},Z_{t_k})  \|  \\
& \le &  C(L_f, C^y_{\ref{difference-estimates for Y and Z}},C_{\ref{betterZ}},\kappa_2, b,\sigma,T,p_0, \delta) \Psi(x) \frac{h^{\frac{1}{2}} }{(s-t_k)^\frac{1}{2}},
\tion 
since for $0\le t <s \le T$ we have by Theorem  \ref{difference-estimates for Y and Z} and Proposition \ref{betterZ} that
\equal \label{f-bounds}
\| f(s, X_s, Y_s,Z_s)- f(t, X_t,Y_t, Z_t)\| &\le& C(L_f,
C^y_{\ref{difference-estimates for Y and Z}},C_{\ref{betterZ}}, b,\sigma,T,p_0) \Psi(x) 
(s-t)^{\frac{1}{2}}.  
 \tionl
For the estimate of ${\tt t}_2$  Lemma \ref{X-Lemma}, Lemma \ref{N-difference-estimate}, 
\eqref{F Lipschitz-new}  and \eqref{norm-weight} yield
\equa
 \|{\tt t}_2(m)\|&=& \|  \ti \e  F(t_m,\ti X^{t_k,X_{t_k}}_{t_m})  \ti N^{t_k}_{t_m} \sigma(t_k,X_{t_k})  -   \ti \e  F(t_m, \ti \cX^{\tau_k,\cX_{\tau_k}}_{\tau_m}) \ti N^{n,\tau_k}_{\tau_m} \sigma(t_{k+1},\cX_{\tau_k})  \| \\
 &\le &  \frac{ C(\kappa_2, \sigma)}{(t_m - t_k)^\half} \pr{\e \ti \e  (F(t_m,\ti X^{t_k,X_{t_k}}_{t_m})  - F(t_m, \ti \cX^{\tau_k,\cX_{\tau_k}}_{\tau_m}))^2 }^{\half}\\
&&+  ( \e   \ti \e   |F(t_m,\ti \cX^{\tau_k,\cX_{\tau_k}}_{\tau_m})-F(t_m,\cX_{\tau_k})|^2    \ti \e  | \ti N^{t_k}_{t_m} \sigma(t_k,X_{t_k})  -    \ti N^{n,\tau_k}_{\tau_m} \sigma(t_{k+1},\cX_{\tau_k}) |^2  )^\half \\
&\le&  C(L_f, c^{2,3}_{\ref{thm1}}, \kappa_2,b,\sigma,T,p_0,\delta) \frac{\Psi(x)}{(T-t_m)^{\frac{1}{2}} }  \frac{h^\frac{1}{4}}{ (t_m -t_k)^\frac{1}{2}}.  
\tion
For ${\tt t}_3$ we use the conditional H\"older inequality, \eqref{F Lipschitz-new}, \eqref{norm-N}  and Lemma \ref{X-Lemma}:
\equa
 \|{\tt t}_3(m)\|&=& \left\|\e_{\tau_k}\left [ [F(t_m,X^n_{t_m} ) - F(t_m,X_{t_m})] N^{n,\tau_k}_{\tau_m} \sigma(t_{k+1},\cX_{\tau_k})   \right ] \right\| \\
 &\le &   \frac{   C(\widehat\kappa_2, \sigma)}{(t_m - t_k)^\half}  \left\|  F(t_m, X^n_{t_m} ) - F(t_m,X_{t_m}) \right\|   \\
 &\le &  C(L_f, c^{2,3}_{\ref{thm1}},b,\sigma,T,p_0,\delta)  \frac{\Psi(x)}{(T-t_m)^{\frac{1}{2}} }  \frac{h^\frac{1}{4}}{ (t_m -t_k)^\frac{1}{2}}. 
\tion
The term  ${\tt t}_4$ can be estimated as follows:
\equa
 \|{\tt t}_4(m)\|&=& \left\|  \e_{\tau_k} \left [ [f(t_m, X_{t_m},Y_{t_m},Z_{t_m})- f(t_{m+1}, X^n_{t_m}, Y^n_{t_m},Z^n_{t_m})]N^{n,\tau_k}_{\tau_m} \sigma(t_{k+1},\cX_{\tau_k})    \right ] \right\| \\
 &\le &  \frac{C(L_f, b,\sigma,T, \delta)}{ (t_m -t_k)^\half} (h^\half + \|  X_{t_m} -X^n_{t_m} \| + \|  Y_{t_m} -Y^n_{t_m} \| +  \| Z_{t_m} - Z^n_{t_m} \|).
 \tion
Finally, for the remaining term of the estimate of    $\| Z_{t_k} -Z^n_{t_k}\|, $ we use 
 \eqref{f-bounds}  and \eqref{norm-weight} to get 
\equa
 \left  \|\e_{t_k} f(s,X_s,Y_s,Z_s) N^{t_k}_s  \, \sigma(t_k,X_{t_k})  \right \| 
 &=&  \|\e_{t_k}[(  f(s,X_s,Y_s,Z_s) - f(s,X_{t_k},Y_{t_k},Z_{t_k})) N^{t_k}_s ]\, \sigma(t_k,X_{t_k}) \| \\
 &\le&  C(L_f, C^y_{\ref{difference-estimates for Y and
     Z}},C_{\ref{betterZ}}, \kappa_2, b,\sigma,T, p_0) \Psi(x). 
 \tion 
 Consequently, from \eqref{Zand-its-approximations}, \eqref{Z-terminal-cond}, the estimates for the remaining term  and for ${\tt t}_1,...,{\tt t}_4$ it follows that 
\equa
 \| Z_{t_k} -Z^n_{t_k}\| 
 &\le &   C_{\ref{discreteZand-wrongZdifference}} \hat \Psi (x)  h^{\frac{\alpha}{2}} +    (C^z_{\ref{no f}})^\half \Psi(x) h^\frac{1}{4} + C(L_f, C^y_{\ref{difference-estimates for Y and
     Z}},C_{\ref{betterZ}},b,\sigma,T, p_0,\kappa_2) \Psi(x) h\\
&&+    C(L_f, C^y_{\ref{difference-estimates for Y and Z}},C_{\ref{betterZ}},\kappa_2,b,\sigma,T,p_0,\delta) \Psi(x) h^{\frac{1}{2}} 
 \int_{t_k}^{T} \frac{ds }{ (s-t_k)^\frac{1}{2}}\\
&&+   C(L_f, c^{2,3}_{\ref{thm1}}, \kappa_2, b,\sigma,T,p_0,\delta) h  \sum_{m=k+1}^{n-1}   \frac{\Psi(x)}{(T-t_m)^{\frac{1}{2}} }  \frac{h^\frac{1}{4}}{ (t_m -t_k)^\frac{1}{2}}  \\
&& +   C(L_f, b,\sigma,T,  \delta )h \sum_{m=k+1}^{n-1}  (\|  Y_{t_m} -Y^n_{t_m} \| +  \| Z_{t_m} - Z^n_{t_m} \|)\frac{1}{ (t_m -t_k)^\half} \\
&\le &  C(C_{\ref{discreteZand-wrongZdifference}},C^z_{\ref{no f}}) \hat \Psi (x) h^{\frac{\alpha}{2}\wedge\frac{1}{4}}  
 +    C(L_f, c^{2,3}_{\ref{thm1}}, C^y_{\ref{difference-estimates for Y and Z}},C_{\ref{betterZ}},\kappa_2,b,\sigma,T,p_0,\delta) \Psi(x)   h^\frac{1}{4}  \\
 && +  C(L_f, b,\sigma,T,  \delta ) \sum_{m=k+1}^{n-1}  (\|  Y_{t_m} -Y^n_{t_m} \| +  \| Z_{t_m} - Z^n_{t_m} \|)\frac{1}{(t_m -t_k)^\half}  h.
\tion

Then we use \eqref{newYdiff} and the above estimate to get
\equa
&& \les \| Y_{t_k} - Y^n_{t_k}\| + \| Z_{t_k} -Z^n_{t_k}\| \\ &\le &
   C(C_{\ref{discreteZand-wrongZdifference}},C^z_{\ref{no f}})\hat  \Psi (x) 
 h^{\frac{\alpha}{2}\wedge\frac{1}{4}}    +    C(L_f,C^y_{\ref{no f}},C^y_{\ref{difference-estimates for Y and
        Z}},  C_{\ref{betterZ}}, c^{2,3}_{\ref{thm1}}, \kappa_2, b, \sigma, T,p_0, \delta) \Psi(x) h^\frac{1}{4}  \\
 && + C(L_f,b,\sigma,T, \delta) \sum_{m=k+1}^{n-1}  (\|  Y_{t_m} -Y^n_{t_m} \| +  \| Z_{t_m} - Z^n_{t_m} \|)\frac{1}{ (t_m -t_k)^\half} h.
\tion
Consequently, {\ce  summarizing the dependencies, there is a $ C=C(b,\sigma,f,g,T,p_0,\delta)$ such that} 
 \equa
 \| Y_{t_k} - Y^n_{t_k}\| + \| Z_{t_k} -Z^n_{t_k}\| &\le& {\ce C }
 \hat \Psi(x)  h^{\frac{\alpha}{2}\wedge\frac{1}{4}}. 
\tion
By Theorem \ref{difference-estimates for Y and Z}  {\ce (note that by Assumption  \ref{Xgeneral}  on $g$ we have $\alpha=1$)} it follows that   
$$\| Y_v- Y^n_v\| \le \| Y_v- Y_{t_k}\|+ \| Y_{t_k}- Y^n_{t_k}\| \le   C^y_{\ref{difference-estimates for Y and Z}}\Psi(x) h^\half +    \hat \Psi(x)  h^{\frac{\alpha}{2}\wedge\frac{1}{4}}.   $$
while Proposition \ref{betterZ} implies that
$$\| Z_v- Z_{t_k} \|  \le C_{\ref{betterZ}}\Psi(x) h^{\frac{1}{2}}, $$  
{\ce and hence we have 
\begin{align*}
  \e_{0,x} |Y_v - Y^n_v|^2   +     \e_{0,x} |Z_v - Z^n_v|^2      \le   C_{\ref{the-result}}  \hat{\Psi}(x)^2h^{ \frac{1}{2} \wedge \alpha}  
\end{align*}
with  $C_{\ref{the-result}}  =  C_{\ref{the-result}}(b,\sigma,f,g,T,p_0,\delta).$}

\section{Some properties of solutions to BSDEs and their associated PDEs  \label{4}}

\subsection{Malliavin weights}
We use the SDE from \eqref{BSDE} started in $(t,x)$, 
\equal \label{process-X}
X^{t,x}_s = x +  \int_t^s b(r,X^{t,x}_r)dr + \int_t^s  \sigma(r, X^{t,x}_r)dB_r,   \quad 0\le t \le s \le T
\tionl
and recall the Malliavin weight and its properties from  \cite[Subsection 1.1 and Remark 3]{GGG}. 
\begin{lemme}  \label{Malliavin-weights} Let  $H: \R \to \R$ be a polynomially bounded Borel function.
If  Assumption  \ref{hypo2} holds and  $X^{t,x}$   is given by \eqref{process-X} 
then setting  $$  G(t,x) := \e H(X_T^{t,x}) $$ implies that 
 $G \in  C^{1,2}([0,T)\times \R ).$ 
Especially it holds for $0 \le t \le r < T$  that 
$$ \partial_x G(r, X_r^{t,x}) =  \e [ H(X_T^{t,x}) N_T^{r,(t,x)} |\cf^t_r ],
$$
where $(\cf^t_r)_{r\in [t,T]}$  is the augmented natural filtration of  $(B^{t,0}_r)_{r \in [t,T]},$  
$$N_T^{r,(t,x)}= \frac{1}{T-r} \int_r^T\frac{\nabla X^{t,x}_s }{\sigma(s,X_s^{t,x}) \nabla X^{t,x}_r  } dB_s,
$$
 and $\nabla X^{t,x}_s$ is given in \eqref{nabla-X}.
Moreover,  for $ q \in (0, \infty)$ there exists a $\kappa_q>0$  such that a.s. 
\equal \label{norm-weight}
 (\e[| N_T^{r,(t,x)}|^q |\cf^t_r ])^\frac{1}{q} \le \frac{\kappa_q}{(T-r)^\frac{1}{2}} \quad \text{ and} \quad \e[ N_T^{r,(t,x)} |\cf^t_r ]=0 \,\, a.s.
 \tionl
and we have 
  $$\|   \partial_x G(r, X_r^{t,x}) \|_{L_p(\PP)} \le \kappa_{q} \frac{\|H(X_T^{t,x}) - \e[ H(X_T^{t,x})|\cf^t_r ]\|_p  }{\sqrt{T-r}}   $$
  for $1<q,p < \infty$ with $\frac{1}{p} + \frac{1}{q} =1.$  
\end{lemme}

\subsection{Regularity of solutions to BSDEs}
The following result  originates from \cite[Theorem 1]{GGG} where  also path dependent cases were included. We formulate it only for our Markovian
setting but use $\PP_{t,x}$  since we are interested in an estimate for all $(t,x) \in  [0,T) \times \R.$ A sketch of a proof of this formulation can be found in \cite{GLL}. 

\begin{thm} \label{difference-estimates for Y and Z}
Let Assumption \ref{hypo2} and  \ref{hypo3} hold.
Then for any $p\in [2,\infty)$ the following assertions are true.

\begin{enumerate}[(i)]
\item    There exists a constant $C^y_{\ref{difference-estimates for Y and Z}} >0$ such that for $0\le t < s \le T$ and $x\in \R,$ \label{one}
\equa
 \| Y_s - Y_t\|_{L_p(\PP_{t,x})} \le C^y_{\ref{difference-estimates for Y and Z}}  \Psi(x) \left ( \int_t^s (T-r)^{\alpha -1}dr \right )^\half,
\tion

\item  there exists a constant $C^z_{\ref{difference-estimates for Y and Z}} >0$ such that for $0\le t <  s<T$ and $x\in \R,$ 
\equa
 \| Z_s - Z_t\|_{L_p(\PP_{t,x})} \le C^z_{\ref{difference-estimates for Y and Z}}  \Psi(x)  \left ( \int_t^s (T-r)^{\alpha -2}dr \right )^\half.
\tion
\end{enumerate}
The constants $ C^y_{\ref{difference-estimates for Y and Z}}$ and
$C^z_{\ref{difference-estimates for Y and Z}}$ depend on 
   $ (L_f, K_f, C_g,c^{1,2}_{\ref{thm1}}, \kappa_q, b,\sigma, T, p_0,p)$, and 
  $\Psi(x)$ is defined in \eqref{Psi}.
\end{thm}

\subsection{Properties of the associated PDE}
The  theorem below collects properties of the solution to the PDE associated to the FBSDE \eqref{BSDE}. For a proof see   
\cite[Theorem 3.2]{ZhangII}, \cite{ZhangIII} and  \cite[ Theorem 5.4]{GLL}.

\begin{thm}\label{thm1}
  Consider the FBSDE \eqref{BSDE} and let Assumptions \ref{hypo2} and  \ref{hypo3} hold.
  Then for the solution $u$ of the associated PDE
 \equal \label{pde}
\left\{ \begin{array}{l}  u_t(t,x) +   \tfrac{\sigma^2(t,x)}{2} u_{xx}(t,x) + b(t,x) u_x(t,x) + f(t,x,u(t,x), \sigma(t,x) u_x(t,x)) =0,\\
 \hspace{30em} t\in [0,T), x\in \R, \\
 u(T,x)=g(x) , \quad x \in \R \end{array}\right . 
 \tionl 
  
   we have
  \begin{enumerate}[(i)]
  \item  $Y_t=u(t,X_t)$ a.s., where $u(t,x)=\e_{t,x} \!\left (g(X_T)+\int_t^T\! f(r,X_r,Y_r,Z_r)dr \right )$ 
    and  $|u(t,x)|\le c^1_{\ref{thm1}} \Psi(x)$
    with $\Psi$ given in \eqref{Psi}, {\ce where $c^{1}_{\ref{thm1}}$ depends  on  $L_f, K_f, C_g, T,p_0$ and on the bounds and  Lipschitz constants of $b$ and $\sigma.$}
    \item 
    \begin{enumerate}
    \item    $\partial_x u$ exists  and is continuous in $[0,T)\times \R,$
    \item  $Z^{t,x}_s= u_x(s,X_s^{t,x})\sigma(s,X_s^{t,x})$ a.s.,
    \item $|u_x(t,x)|\le
      \frac{ c^2_{\ref{thm1}} \Psi(x)}{(T-t)^{\frac{1-\alpha}{2}}},$
    \end{enumerate}
 {\ce  where $c^{2}_{\ref{thm1}}$ depends on  $L_f, K_f, C_g,T,p_0,\kappa_2= \kappa_2(b,\sigma,T,\delta)$  and on the bounds and  Lipschitz constants of $b$ and $\sigma,$ 
and hence  $c^{2}_{\ref{thm1}} =  c^{2}_{\ref{thm1}}(L_f, K_f, C_g,b,\sigma, T,p_0,\delta).$ }
  \item 
    \begin{enumerate}
   \item $\partial^2_x u$  exists and
    is continuous in $[0,T)\times \R,$
    \item $|\partial^2_x u(t,x)|\le
      \frac{  c^3_{\ref{thm1}} \Psi(x)}{(T-t)^{1-\frac{\alpha}{2}}},$
    \end{enumerate}
    {\ce where $c^{3}_{\ref{thm1}}$ depends on  $L_f, C_g,T,p_0,\kappa_2= \kappa_2(b,\sigma,T,\delta),C^y_{\ref{difference-estimates for Y and Z}},C^z_{\ref{difference-estimates for Y and Z}}$ and on the bounds and Lipschitz constants of $b$ and $\sigma,$  and hence  $c^{3}_{\ref{thm1}} =  c^{3}_{\ref{thm1}}(L_f, K_f, C_g,b,\sigma, T,p_0,\delta).  $
}
  \end{enumerate}
\end{thm}

 Using Assumption \ref{Xgeneral} we are now in the position to improve the bound on $\| Z_s - Z_t\|_{L_p(\PP_{t,x})}$ given in Theorem \ref{difference-estimates for Y and Z}. 

\begin{prop} \label{betterZ}
If  Assumption \ref{Xgeneral} holds, then there exists a constant $C_{\ref{betterZ}} >0$ such that for $0\le t < s \le T$ and $x\in \R,$
$$\| Z_s - Z_t\|_{L_p(\PP_{t,x})}\leq C_{\ref{betterZ}} \Psi(x)
(s-t)^{\frac{1}{2}},$$
 where $C_{\ref{betterZ}}$ depends on $c^{2,3}_{\ref{thm1}}, b,\sigma,f,g,T,p_0,p,$ {\ce and hence $C_{\ref{betterZ}}=  C_{\ref{betterZ}}(b,\sigma, f,g,T,p_0,p,\delta).$}
\end{prop}

\begin{proof}
From $Z^{t,x}_s= u_x(s,X_s^{t,x})\sigma(s,X_s^{t,x})$ and $\nabla Y_s^{t,x} = \partial_x u(s,X_s^{t,x}) =u_x(s,X_s^{t,x})\nabla X_s^{t,x}$ we conclude
\equal  \label{Z-as-fraction}
Z^{t,x}_s = \frac{\nabla Y^{t,x}_s}{\nabla X^{t,x}_s} \sigma(s,X^{t,x}_s), \quad 0 \leq t \le s \leq T.
\tionl
It is well-known (see e.g. \cite{ElKarouiPengQuenez}) that  the solution $\nabla Y$ of the linear BSDE 
\begin{align}\label{nabla-bsde-variational}
\nb Y_s = g'(X_T) \nb X_T + \int_{s}^T f_x(\Theta_r) \nb X_r + f_y(\Theta_r) \nb Y_r + f_z(\Theta_r) \nb Z_r dr - \int_{s}^T \nb Z_r dB_r, \quad 0 \leq s \leq T,
\end{align}
can be represented as 
\begin{align} \label{fraction-computed}
\frac{\nb Y_s}{\nb X_s} & = \e_s \Big[g'(X_T) \nb X_T \Gamma^s_T + \int_{s}^T f_x(\Theta_r) \nb X_r  \Gamma^s_rdr \Big] \frac{1}{\nb X_s}  \notag\\
 & = \ti \e\Big[g'(\ti X^{s, X_s}_T) \nb \ti X^{s,X_s}_T \ti \Gamma^{s,X_s}_T + \int_{s}^T  f_x(\ti \Theta^{s,X_s}_r) \nb \ti X^{s,X_s}_r \ti \Gamma^{s,X_s}_rdr \Big], \quad 0 \leq t \le s \leq T,
\end{align}
where    $\Theta_r := (r,X_r,Y_r,Z_r)$  and $\Gamma^s$  denotes the adjoint process  given by 
$$\Gamma^s_r = 1 + \int_s^r f_y(\Theta_u) \Gamma^s_u du + \int_s^r f_z(\Theta_u) \Gamma^s_u dB_u, \quad   s \le r  \leq T,$$
and
$$\ti \Gamma^{t,x}_s = 1 + \int_{t}^s f_y(\ti \Theta^{t,x}_r) \ti \Gamma^{t,x}_r dr + \int_{t}^s f_z(\ti \Theta^{t,x}_r) \ti \Gamma^{t,x}_r d \ti B_r, \quad t \leq s \leq T, \,x \in \R$$
 where $\ti B$ denotes an independent copy of $B$.
 Notice that $\nb X^{t,x}_t=1,$ so that
\begin{align*}
\frac{\nb Y^{t,x}_t}{\nb X^{t,x}_t} = \nb Y^{t,x}_t
 & = \ti \e\Big[g'(\ti X^{t, x}_T) \nb \ti X^{t,x}_T \ti \Gamma^{t,x}_T + \int_{t}^T  f_x(\ti \Theta^{t,x}_r) \nb \ti X^{t,x}_r \ti \Gamma^{t,x}_rdr \Big].
\end{align*}
Then, by \eqref{Z-as-fraction},
\equa
\| Z_s - Z_t\|_{L_p(\PP_{t,x})}\le  C(\sigma) \bigg [  \bigg\|\frac{\nb Y_s}{\nb X_s}   - \frac{\nb Y_t}{\nb X_t} \bigg\|_{L_{p}(\PP_{t,x})} \!\!\!\!\!+   \|\nb Y_t\|_{L_{2p}(\PP_{t,x})} [(s-t)^\half \!+ \| X^{t,x}_s -x\|_{L_{2p}(\PP_{t,x})}] \bigg ]. 
\tion
Since $(\nabla Y_s, \nabla Z_s)$ is the solution to the linear BSDE  \eqref{nabla-bsde-variational} 
with bounded $f_x, f_y, f_z,$ we have that  $\|\nb Y_t\|_{L_{2p}(\PP_{t,x})} \le  C(b,\sigma,f,g,T,p).$   Obviously, $ \| X^{t,x}_s -x\|_{L_{2p}(\PP_{t,x})} \le  C(b,\sigma,T,p) (s-t)^\half.$
So it remains to show that 
$$\bigg\|\frac{\nb Y_s}{\nb X_s}   - \frac{\nb Y_t}{\nb X_t}
\bigg\|_{L_{p}(\PP_{t,x})}  \le C\Psi(x)(s-t)^\half.$$
 We intend to use \eqref{fraction-computed} in the following.  There is a certain degree of freedom how to connect $B$ and $\ti B$ in order to compute conditional expectations. Here, unlike  in \eqref{independent-B}, we define  
the processes 
$$ B'_u= B_{u\wedge s} + \ti B_{u \vee s}- \ti B_s \quad \text{and} \quad B''_u= B_{u\wedge t} + \ti B_{u \vee t}- \ti B_t, \quad u\ge 0  $$
as driving Brownian motions for  $\tfrac{\nb Y_s}{\nb X_s}$ and  $\tfrac{\nb Y_t}{\nb X_t},$ respectively.
This will especially simplify the estimate for $ \ti \e |\ti \Gamma^{s,X_s}_T -\ti \Gamma^{t,x}_T|^ q $ below.
From the above relations we get  for ($X_s:=X^{t,x}_s $)
\begin{align*}
\bigg\|\frac{\nb Y_s}{\nb X_s}   - \frac{\nb Y_t}{\nb X_t} \bigg\|_{L_{p}(\PP_{t,x})} 
 & \leq \left \| \ti \e \Big[g'(\ti X^{s, X_s}_T) \nb \ti X^{s,X_s}_T \ti \Gamma^{s,X_s}_T  - g'(\ti X^{t,x}_T) \nb \ti X^{t,x}_T \ti \Gamma^{t,x}_T  \Big] \right \|_p\\
& \quad +  \int_t^s \left\| \ti \e \Big[ f_x(\ti \Theta^{t,x}_r) \nb \ti X^{t,x}_r \ti  \Gamma^{t,x}_r\Big]\right\|_p dr \\
&  \quad  +\left\| \int_s^T \ti \e \Big[ f_x(\ti \Theta^{s,X_s}_r) \nb \ti X^{s,X_s}_r\ti \Gamma^{s,X_s}_r
- f_x(\ti \Theta^{t,x}_r) \nb \ti X^{t,x}_r \ti \Gamma^{t,x}_r \Big] dr \right\|_p\\
& =: J_1 + J_2 + J_3.
\end{align*}
Since $g'$ is Lipschitz continuous and of polynomial growth, the estimate
$J_1\le  C(b,\sigma,g,T,p) \Psi(x) (s-t)^\half $ follows by H\"older's inequality and  the $L_q$ -boundedness  for any  $q >0$ of all the factors, as well as  
from the estimates for $  \ti X^{s, X_s}_T - \ti X^{t,x}_T$ and $\nb \ti X^{s,X_s}_T - \nb \ti X^{t,x}_T $  like in Lemma \ref{X-Lemma}.  
For the $\Gamma$ differences we first apply  the inequalities of H\"older and BDG:
\equa
  \ti \e |\ti \Gamma^{s,X_s}_T -\ti \Gamma^{t,x}_T|^ q &\le C(T,q) & \!\!\!\bigg [ (s-t)^{q-1} \ti \e \int_t^s |f_y(\ti \Theta^{s,X_s}_r) \ti \Gamma^{s,X_s}_r|^q dr 
  + \ti \e \bigg(\int_t^s |f_z (\ti \Theta^{s,X_s}_r) \ti \Gamma^{s,X_s}_r |^2 dr \bigg)^\frac{q}{2}  \\
&& +  \ti \e \int_s^T |f_y(\ti \Theta^{s,X_s}_r) \ti \Gamma^{s,X_s}_r - f_y(\ti \Theta^{t,x}_r) \ti \Gamma^{t,x}_r|^q dr \\
&&+  \ti \e \bigg( \int_s^T |f_z(\ti \Theta^{s,X_s}_r) \ti \Gamma^{s,X_s}_r- f_z(\ti \Theta^{t,x}_r) \ti \Gamma^{t,x}_r|^2 dr \bigg)^\frac{q}{2} \bigg ]
 \tion
 Since  $f_y$ and $f_z$ are bounded we have $\ti \e | \ti \Gamma^{s,X_s}_r|^q + \ti
 \e  |\ti \Gamma^{t,x}_r |^q \le  C(f,T,q). $
 Similar to \eqref{F Lipschitz-new}, since $f_x, f_y,f_z$ are Lipschitz continuous w.r.t.~the space variables,  
 \equa
 |f_x(\ti \Theta^{s,X_s}_r)  - f_x(\ti \Theta^{t,x}_r) |&=& |f_x(r, \ti X_r^{s,X_s}, u(r,  \ti X_r^{s,X_s}), \sigma(r,  \ti X_r^{s,X_s}) u_x(r,  \ti X_r^{s,X_s}) ) \\
&& \quad \quad \quad  \quad \quad \quad  - f_x(r, \ti X_r^{t,x}, u(r,  \ti X_r^{t,x}), \sigma(r,  \ti X_r^{t,x}) u_x(r,  \ti X_r^{t,x}) )| \\
&\le &  C(c^{2,3}_{\ref{thm1}}, \sigma, f, T) (1+|\ti X_r^{s,X_s}|^{p_0+1}+|\ti X_r^{t,x}|^{p_0+1}) \frac{|\ti X_r^{s,X_s}-\ti X_r^{t,x}|}{(T-r)^{\frac{1}{2}}},
 \tion
so that  Lemma \ref{X-Lemma} yields  
$$\ti \e   |f_x(\ti \Theta^{s,X_s}_r)  -
f_x(\ti \Theta^{t,x}_r) |^q \le   C(c_{\ref{thm1}}^{2,3}, b,\sigma, f, T,p_0,q) (1+|X_s|^{p_0+1}+|x|^{p_0+1})^q  \frac{ |X_s-x|^q + |s-t|^\frac{q}{2}}{ (T-r)^{\frac{1}{2}}}.$$ 
The same holds for  $|f_y(\ti \Theta^{s,X_s}_r)  - f_y(\ti \Theta^{t,x}_r) |$ and  $|f_z(\ti \Theta^{s,X_s}_r)  - f_z(\ti \Theta^{t,x}_r) |.$ Applying these inequalities  and Gronwall's lemma, we arrive at 
 \equa
  \| \ti \e [\ti \Gamma^{s,X_s}_T -\ti \Gamma^{t,x}_T] \|_p &\le&   C(c_{\ref{thm1}}^{2,3}, b,\sigma,f,g, T,p_0,p) \Psi(x) |s-t|^\frac{1}{2} 
\tion
for $p> 0.$

For $J_2\le  C (t-s)$ it is enough to realise that the integrand is bounded. The estimate for $J_3$ follows similarly to that of $J_1.$
\end{proof}

\subsection{Properties of the solution to the finite difference equation}
 Recall the definition of $ \mathcal{D}^n_m$ given in \eqref{Dm}.  By  \eqref{Xn},
\equal \label{one-step-for-X}
 X_{t_{m+1}}^{n,t_m,x} = x+   h b(t_{m+1},x)  + \sqrt{h} \sigma(t_{m+1},x) \ep_{m+1},
\tionl
so that  
\equal \label{T-applied-to-un}T_{_{m+1,\pm}} u^n(t_{m+1}, X_{t_{m+1}}^{n,t_m,x }) = u^n(t_{m+1},  
x+   h b(t_{m+1},x)  \pm \sqrt{h} \sigma(t_{m+1},x)). 
\tionl     
While for the solution to the PDE  \eqref{pde}  one can observe   in   Theorem \ref{thm1} the well-known smoothing property  which implies that $u$ 
is differentiable on $[0,T)\times \rset$ even though $g$ is only H\"older continuous, in following proposition,  for  the solution $u^n$ to the finite difference equation we have to require
from $g$ the  same regularity as we want for $u^n.$  
\begin{prop} \label{u-discrete}
Let Assumption \ref{Xgeneral}  hold and assume that  $u^n$ is a solution of 
\equal \label{discrete-pde}
&& \hspace{-4em}u^n(t_m,x) -h f(t_{m+1}, x,u^n(t_m,x), \mathcal{D}^n_{m+1} u^n(t_{m+1},X_{t_{m+1}}^{n,t_m,x})) \notag \\&=&
 \half [T_{_{m+1,+}}  u^n(t_{m+1}, X_{t_{m+1}}^{n,t_m,x }) + T_{_{m+1,-}} u^n(t_{m+1}, X_{t_{m+1}}^{n,t_m,x })], \quad  m = 0, \dots, n-1,
 \tionl 
with terminal condition $u^n(t_n,x)= g(x).$
Then, for sufficiently small $h,$ the map   $x \mapsto u^n(t_m,x)$  is $C^2,$  and    it holds 
$$
|u^n(t_m,x)| +  |u_x^n(t_m,x)| \le C_{u^n\!,1}\, \Psi(x), \quad  |u_{xx}^n(t_m,x)| \le C_{u^n\!,2}\, \Psi^2(x)$$  and 
\equal \label{u-xx-alpha}
|u_{xx}^n(t_m,x) - u_{xx}^n(t_m, \bar x)|  \le C_{u^n\!,3} \,(1+|x|^{6p_0+7} +
|\bar x|^{6p_0+7})|x-\bar x|^\alpha,
\tionl
 uniformly in $ m=0,\dots,n-1$. The constants $C_{u^n\!,1}$, $C_{u^n\!,2}$ and
 $C_{u^n\!,3}$ depend on the bounds of $f,g,b,\sigma$ and their derivatives and on $T$ and $p_0$.
\end{prop}
\begin{proof}
{\bf Step 1.}
From \eqref{discrete-pde}, since $g$ is $C^2$ and $f_y$ is bounded, for sufficiently small $h$ we conclude   by  induction (backwards in time)  that  $u^n_x(t_m,x)$ exists for $m=0,...,n-1,$ and that it holds
\equa
u_x^n(t_m,x)&=&h f_x(t_{m+1}, x,u^n(t_m,x), \mathcal{D}^n_{m+1} u^n(t_{m+1},X_{t_{m+1}}^{n,t_m,x})) \\
&&+h f_y(t_{m+1}, x,u^n(t_m,x), \mathcal{D}^n_{m+1} u^n(t_{m+1},X_{t_{m+1}}^{n,t_m,x})) \, u_x^n(t_m,x) \\
&&+h f_z(t_{m+1}, x,u^n(t_m,x), \mathcal{D}^n_{m+1} u^n(t_{m+1},X_{t_{m+1}}^{n,t_m,x})) \, \partial_x\mathcal{D}^n_{m+1} u^n(t_{m+1},X_{t_{m+1}}^{n,t_m,x}) \\
&& +\tfrac{1}{2}\Big ( \partial_xT_{_{m+1,+}}  u^n(t_{m+1}, X_{t_{m+1}}^{n,t_m,x }) + \partial_xT_{_{m+1,-}} u^n(t_{m+1}, X_{t_{m+1}}^{n,t_m,x} )\Big).
 \tion
Similarly one can show that $u^n_{xx}(t_m,x)$ exists and  solves the derivative of the previous equation. \bigskip

{\bf Step 2.} 
As stated in the proof of Proposition \ref{discreteZand-wrongZdifference}, the finite difference equation  \eqref{discrete-pde} is the associated equation to   \eqref{discrete-BSDE-sums}
in the sense that we have the representations \eqref{representations}.
We will use that  $u^n(t_m,x) =Y_{t_m}^{n,t_m,x}$ and exploit the BSDE
\equal \label{level0}
  Y_{t_m}^{n,t_m,x}\!\!\!&=&\!\!\!  g(X_T^{n,t_m,x}) +\int_{(t_m,T]} f(s,X^{n,t_m,x}_{s-},Y^{n,t_m,x}_{s-},Z^{n,t_m,x}_{s-})d[B^n]_s  \notag \\ && \hspace{15em} - \int_{(t_m,T]} Z^{n,t_m,x} _{s-} dB^n_s, 
  \tionl
where we will drop the superscript   $t_m,x$ from now on. 
For   $u^n_x(t_m,x)$  we will  consider
\equal \label{level1}
 \nabla Y^n_{t_m}:= \partial_x Y_{t_m}^{n} &=& g'(X_T^{n}) \partial_x X_T^{n}  +\int_{(t_m,T]}  f_x \partial_x  X_{s-}^{n} + f_y   \partial_x Y_{s-}^{n}  + f_z \partial_x Z_{s-}^{n}d[B^n]_s \notag \\
 && \hspace{15em} -\int_{(t_m,T]} \partial_x Z_{s-}^{n} dB^n_s. 
\tionl 
Similarly as in the proof of  \cite[Theorem 3.1]{MaZhang} the BSDE  \eqref{level1}  can be derived from \eqref{level0} as a limit of  difference quotients w.r.t.~$x.$
 Notice that the generator of  \eqref{level1} is random but has the same Lipschitz constant and linear growth bound as $f.$  
 Assumption  \ref{Xgeneral} allows us to find a $p_0\ge 0$ and a $K>0$ such that $$|g(x)| +|g'(x)|+|g''(x)| \le K(1+|x|^{p_0+1}) =\Psi(x).$$  In order to  get  estimates simultaneously for   \eqref{level0} and \eqref{level1}
 we show the following lemma. 
\begin{lemme} \label{polynomial-estimates}
We fix $n$ and  assume a BSDE
\equal \label{BSDE-for-both}
  \sY_{t_k}&=&  \xi^n +\int_{(t_k,T]}  {\sf f}(s,\sX_{s-},\sY_{s-},\sZ_{s-})d[B^n]_s - \int_{(t_k,T]} \sZ _{s-} dB^n_s, \quad m\le k\le n, 
  \tionl
with  $\xi^n =  g(X_T^{n,t_m,x}) $  or   $\xi^n =  g'(X_T^{n,t_m,x}) \partial_x X_T^{n,t_m,x} $  and $\sX_s:= X_s^{n,t_m,x}$ or  
$\sX_s:= \partial_x X_s^{n,t_m,x}$ such that ${\sf f}:\Omega \times [0,T] \times \R^3 \to \R $ is measurable  and satisfies 
\equal \label{f-growth}
  |{\sf f}(\omega,t,x,y,z)  - {\sff}(\omega,t,x',y',z')  |& \le &L_f ( |x-x'| + |y-y'| + |z-z'|), \notag \\
  |{\sf f}(\omega,t,x,y,z) |& \le& (K_f+L_f) (1+ |x| + |y| + |z|).
  \tionl
Then for any $p \ge 2,$ 
\begin{enumerate}[(i)]
\item \label{polynomial-a-priori} $ \e |\sY_{t_k}|^p +  \tfrac{\gamma_p}{4}   \e \int_{(t_k,T]}   |\sY_{s-} |^{p-2}  |\sZ_{s-}|^2 d[B^n]_s \le  C\Psi^p(x),$ \quad 
for   \,\,$k=m,...,n$ \,\, and some $\gamma_p>0,$
\item  \label{supY}   $ \e  \sup_{t_m < s \le T}|\sY_{s-}|^p  \le C \Psi^p(x), $
\item  \label{Z-power} $ \e \Big ( \int_{(t_m,T]}  |\sZ_{s-}|^2 d[B^n]_s \Big)^{\frac{p}{2}} \le   C\Psi^p(x),$
\end{enumerate}
for some constant $ C=C(b,\sigma, f,g,T,p,p_0)$.
\end{lemme}
\begin{proof}
\eqref{polynomial-a-priori}
 By It\^o's formula   (see  \cite[Theorem 4.57]{Jacod/Shiryaev})  we get for $p\ge 2$
\equal \label{ito-of-pth-power}
|\sY_{t_k}|^p& = &  | \xi^n|^p - p  \int_{(t_k,T]} \sY_{s-} |\sY_{s-} |^{p-2} \sZ_{s-} dB^n_s  + p\int_{(t_k,T]} \sY_{s-} |\sY_{s-} |^{p-2}   \sff(s,\sX_{s-},\sY_{s-}, \sZ_{s-})d[B^n]_s \notag\\
&& - \sum_{s \in (t_k,T]}  [|\sY_s|^p  - |\sY_{s-}|^p - p \sY_{s-} |\sY_{s-} |^{p-2}  (\sY_s-\sY_{s-})].
\tionl 
Following the proof of \cite[Proposition 2]{KrusePopier}  (which is carried out there in the  L\'evy process  setting but can be done also for martingales  with jumps, like $B^n$) we can use the estimate
\equa 
 - \sum_{s \in (t_k,T]} [ |\sY_s|^p  - |\sY_{s-}|^p - p \sY_{s-} |\sY_{s-} |^{p-2}  (\sY_s-\sY_{s-})] 
\le - \gamma_p   \sum_{s \in (t_k,T]}   |\sY_{s-}|^{p-2}  (\sY_s-\sY_{s-})^2 
\tion
where $\gamma_p >0$ is computed in  \cite[Lemma A4]{Yao}. Since $$ \sY_{t_{\ell+1}} - \sY_{{t_{\ell+1}}-}=\sff(t_{\ell+1},\sX_{t_\ell},\sY_{t_\ell}, \sZ_{t_\ell})h 
- \sZ_{t_\ell} \sqrt{h} \eps_{\ell+1}$$
 we have
\equa 
&&\less - \sum_{s \in (t_k,T]} [ |\sY_s|^p  - |\sY_{s-}|^p - p \sY_{s-} |\sY_{s-} |^{p-2}  (\sY_s-\sY_{s-})] \\
&\le&   - \gamma_p \,\sum_{\ell=k}^{n-1}   |\sY_{t_\ell} |^{p-2} \, \Big (\sff(t_{\ell+1},\sX_{t_\ell},\sY_{t_\ell}, \sZ_{t_\ell})h - \sZ_{t_\ell} \sqrt{h} \eps_{\ell+1} \Big)^2  \\
&=& -  \gamma_p \,  h \int_{(t_k,T]}   |\sY_{s-} |^{p-2} \,  \sff^2(s,\sX_{s-},\sY_{s-},\sZ_{s-})d[B^n]_s   - \gamma_p     \int_{(t_k,T]}   |\sY_{s-} |^{p-2}  |\sZ_{s-}|^2 d[B^n]_s  \\
&&   +  2 \gamma_p  \,    \int_{(t_k,T]}    |\sY_{s-} |^{p-2} \, \sff(s,\sX_{s-},\sY_{s-}, \sZ_{s-})  \sZ_{s-}(B^n_s-B^n_{s-})  d[B^n]_s.
\tion
  Hence we get from \eqref{ito-of-pth-power}
\equa
 |\sY_{t_k}|^p
&\le &| \xi^n|^p-  p  \int_{(t_k,T]} \sY_{s-} |\sY_{s-} |^{p-2} \sZ_{s-} dB^n_s  + p\int_{(t_k,T]} \sY_{s-} |\sY_{s-} |^{p-2} \,  \sff(s,\sX_{s-},\sY_{s-}, \sZ_{s-})d[B^n]_s  \notag\\
&&  - \gamma_p     \int_{(t_k,T]}   |\sY_{s-} |^{p-2} \, |\sZ_{s-}|^2 d[B^n]_s  \\
&&   +  2 \gamma_p  \,    \int_{(t_k,T]}    |\sY_{s-} |^{p-2} \, \sff(s,\sX_{s-},\sY_{s-}, \sZ_{s-})  \sZ_{s-}(B^n_s-B^n_{s-})  d[B^n]_s.
\tion
 From  Young's inequality and \eqref{f-growth} we conclude that there is a $ c'= c'(p,K_f,L_f, \gamma_p)>0$ such that
 \equa
 p|\sY_{s-} |^{p-1}\, |\sff(s,\sX_{s-},\sY_{s-},\sZ_{s-})| 
\le \tfrac{\gamma_p}{4} |\sY_{s-} |^{p-2} \, |\sZ_{s-}|^2 +  c'(1+| \sX_{s-}|^p   +   |\sY_{s-} |^p),    
\tion
and  for   $ \sqrt{h}   < \tfrac{1}{8  (L_f + K_f)}$  we find a $c'' =c''(p, L_f, K_f, \gamma_p )>0$ such that 
$$2 \gamma_p  \, \sqrt{h}    |\sY_{s-} |^{p-2} \,  |\sff(s,\sX_{s-},\sY_{s-},\sZ_{s-})|  |  \sZ_{s-}| \le  \tfrac{ \gamma_p}{4}  |\sY_{s-} |^{p-2} \, |  \sZ_{s-}|^2+ c'' \,(1+|\sX_{s-}|^p +|\sY_{s-}|^p).$$
Then for $c=c'+c''$ we have
\equal \label{Y-p-estimate}
 |\sY_{t_k}|^p
&\le &| \xi^n|^p - p  \int_{(t_k,T]} \sY_{s-} |\sY_{s-} |^{p-2}\, \sZ_{s-} dB^n_s   + c\int_{(t_k,T]} 1+| \sX_{s-}|^p   +   |\sY_{s-} |^p d[B^n]_s \notag\\
&&- \tfrac{\gamma_p}{2}     \int_{(t_k,T]}   |\sY_{s-} |^{p-2} \, |\sZ_{s-}|^2 d[B^n]_s. 
\tionl
By standard methods, approximating  the  terminal condition and the generator by bounded functions,  it follows  that for any $a>0$    
$$\e \sup_{t_k \le s\le T} |\sY_s|^a < \infty \quad \text{ and } \quad \e \left (\int_{(t_k,T]} |\sZ_{s-}|^2 d[B^n]_s \right)^{\frac{a}{2}} < \infty.$$ 
Hence $ \int_{(t_k,T]} \sY_{s-} |\sY_{s-} |^{p-2} \sZ_{s-} dB^n_s$ has expectation zero.  Taking  the expectation in \eqref{Y-p-estimate} yields 
\equal \label{apriori}
 \e |\sY_{t_k}|^p +  \tfrac{\gamma_p}{2}   \e \int_{(t_k,T]}  \!\! |\sY_{s-}|^{p-2}  |\sZ_{s-}|^2 d[B^n]_s  
\le  \e| \xi^n|^p  + c \e \int_{(t_k,T]} \!\! 1+| \sX_{s-}|^p   +   |\sY_{s-} |^p d[B^n]_s. 
\tionl
Since  $\e| \xi^n|^p$  and $ \e \int_{(t_k,T]} 1+| \sX_{s-}|^p d[B^n]_s$
  are polynomially bounded in $x$, Gronwall's lemma gives
\equa
&& \less  \| \sY_{t_k}\|_p \le   C(b,\sigma, f,g, T,p,p_0) (1 + |x|^{ p_0+1}), \quad  k=m,...,n, 
\tion
and inserting this into \eqref{apriori} yields
\equa
  \Big (\e \int_{(t_k,T]}   |\sY_{s-} |^{p-2}  |\sZ_{s-}|^2 d[B^n]_s \Big )^\frac{1}{p}  \le   C(b,\sigma,f,g,T,p,p_0) (1 + |x|^{ p_0+1}), \quad k=m,...,n-1.
\tion

\eqref{supY}
 From \eqref{Y-p-estimate} we derive by the  inequality of BDG and Young's inequality that for $t_m \le t_k \le T$
\equa
&& \les \e  \sup_{t_k < s \le T}|\sY_{s-}|^p \\
&\le &   \e| \xi^n|^p  + C(p)  \e \left ( \int_{(t_k,T]} |\sY_{s-} |^{2p-2} |\sZ_{s-} |^2 d[B^n]_s\right )^\half  +c\e \int_{(t_k,T]}1+| \sX_{s-}|^p   +   |\sY_{s-} |^p d[B^n]_s \notag\\
&\le &  \e| \xi^n|^p  +  c \e \int_{(t_k,T]} 1+| \sX_{s-}|^p  d[B^n]_s   +   C(p)  \e  \left [ \sup_{t_k < s \le T}  |\sY_{s-} |^{\frac{p}{2}} \left ( \int_{(t_k,T]} |\sY_{s-} |^{p-2} |\sZ_{s-} |^2 d[B^n]_s\right )^\half  \right ]\notag\\
&& + c \e \int_{(t_k,T]}   |\sY_{s-} |^pd[B^n]_s \\
&\le &   \e| \xi^n|^p  +  c \e \int_{(t_k,T]} \!\!\! 1+| \sX_{s-}|^p d[B^n]_s   +   C(p)  \e  \int_{(t_k,T]} |\sY_{s-} |^{p-2} |\sZ_{s-} |^2 d[B^n]_s \\ 
&&  + \e \sup_{t_k < s \le T}  |\sY_{s-} |^p( \tfrac{1}{4} + c (T-t_k)).   
\tion
 We assume that $h$ is sufficiently small so that we find a $t_k$ with  $c (T-t_k) < \tfrac{1}{4}.$  We rearrange the inequality to have  $ \e \sup_{t_k < s \le T}  |\sY_{s-} |^p$ on the l.h.s., and  from  \eqref{polynomial-a-priori} we conclude  that
\equa
 \e  \sup_{t_k < s \le T}|\sY_{s-}|^p &\le&   2 \e | \xi^n|^p  +  2c \e \int_{(t_k,T]} \!\!\! 1+| \sX_{s-}|^p d[B^n]_s   +   2C(p)  \e  \int_{(t_k,T]} |\sY_{s-} |^{p-2} |\sZ_{s-} |^2 d[B^n]_s \\ 
 &\le &  C(b,\sigma, f,g, T,p,p_0) (1+|x|^{(p_0+1)p}).
 \tion

Now we may repeat the above step for  $\e  \sup_{t_\ell < s \le t_k}|\sY_{s-}|^p$ with $c (t_k-t_\ell) < \tfrac{1}{4}$  and $\xi^n=\sY_T$ replaced by $\sY_{t_k},$ and continue doing so until we eventually  get assertion
 \eqref{supY}. \\
 \eqref{Z-power}
We proceed from \eqref{BSDE-for-both},
\equa
  \sup_{k\le \ell\le n} \Big | \int_{(t_\ell,T]}  \sZ_{s-} dB^n_s \Big|^p
  \le C(p) \bigg (  | \xi^n|^p +  \sup_{k\le \ell \le n} |\sY_{t_\ell}|^p +
  \Big (\int_{(t_k,T]} | {\sf f}(s,\sX_{s-},\sY_{s-}, \sZ_{s-}) | \,d[B^n]_s\Big )^p \bigg),
\tion
so that by  \eqref{f-growth} and  the inequalities of BDG and H\"older we have that
\equa
&& \les  \e \Big ( \int_{(t_k,T]}  |\sZ_{s-}|^2 d[B^n]_s \Big)^{\frac{p}{2}} \\
&\le& C(p) \Big ( \e | \xi^n|^p +  \e \sup_{k\le \ell\le n} | \sY_{t_\ell}|^p \Big ) +C(p,L_f,K_f) \e \left (\int_{(t_k,T]}  1+|\sX_{s-}|+| \sY_{s-}|d[B^n]_s\right )^p \\
 &&+C(p,L_f,K_f) (T-t_k)^{\frac{p}{2}} \e \left (\int_{(t_k,T]}   |\sZ_{s-}|^2d[B^n]_s\right )^{\frac{p}{2}}.
 \tion
Hence  for $C(p,L_f,K_f)  (T-t_k)^{\frac{p}{2}} <\half$ we derive from
assertion  \eqref{supY}  and from the growth properties of the other terms
that
\begin{align}\label{eq26}
 \e \Big ( \int_{(t_k,T]}  |{\sZ_{s-}}|^2 d[B^n]_s \Big)^{\frac{p}{2}} \le  C(b,\sigma, f,g,T,p,p_0) (1+|x|^{(p_0+1)p}).
\end{align}
Repeating this procedure eventually yields \eqref{Z-power}.
\end{proof}
{\bf  Step 3.}
Applying Lemma  \ref{polynomial-estimates}  to \eqref{level0} and
\eqref{level1}  we see that for all $m =0,...,n$ we have
$$ |u^n(t_m, x)| = |Y_{t_m}^{n, t_m,x}  | = ( \e (Y_{t_m}^{n, t_m,x} )^2)^\frac{1}{2} \le  C(b,\sigma, f,g,T,p_0) (1+|x|^{p_0+1})$$ 
and 
\equal \label{ux-bound}|u^n_x(t_m, x)|   = ( \e ( \partial_xY_{t_m}^{n, t_m,x})^2 )^\frac{1}{2} \le  C(b,\sigma, f,g,T,p_0) (1+|x|^{p_0+1}).
\tionl

Our  next aim is to show that  $u^n_{xx}(t_m,x)$  is locally Lipschitz in $x.$  We first show that   $u^n_{xx}(t_m,x)$  has polynomial growth.
We introduce the BSDE which describes $u^n_{xx}(t_m,x)$ and  denote  for simplicity  
$$f(t,x_1,x_2,x_3):= f(t,x,y,z)   \quad  \text{and } \quad D^a :=  \partial_{x_1}^{i_1} 
\partial_{x_2}^{i_2}\partial_{x_3}^{i_3}  \quad \text{with} \quad a :=(i_1,i_2,i_3)$$ and consider
\equal \label{level2}
 \partial_x^2 Y_{t_m}^{n}&=&g''(X_T^{n}) (  \partial_x X_T^{n})^2 + g'(X_T^{n}) \partial_x^2  X_T^{n}  \notag\\
&& +\int_{(t_m,T]}  \sum_{\substack{a \in \{0,1,2\}^3 \\ i_1+i_2+i_3=2}}   (D^a f)(s,X_{s-}^{n},Y^{n}_{s-},Z^{n}_{s-}) (\partial_x   X_{s-}^{n})^{i_1}  (\partial_x Y^{n}_{s-})^{i_2}    (\partial_x Z^{n}_{s-})^{i_3}   d[B^n]_s  \notag  \\
&&+  \int_{(t_m,T]}  \sum_{\substack{a  \in \{0, 1\}^3\\ i_1+i_2+i_3=1}}   (D^a f)(s,X_{s-}^{n},Y^{n}_{s-},Z^{n}_{s-}) (\partial_x^2  X_{s-}^{n})^{i_1}  (\partial_x^2 Y^{n}_{s-})^{i_2}    (\partial_x^2 Z^{n}_{s-})^{i_3}   d[B^n]_s    \notag \\
&&-\int_{(t_m,T]} \partial_x^2  Z^n_{s-} dB^n_s.
\tionl 
We denote the generator of this BSDE by $\hat f$ and notice that it  is of the structure   
$$ \hat f(\omega, t,x,y,z) = f_0(\omega, t) + f_1(\omega, t) x+ f_2(\omega, t) y+ f_3(\omega, t) z.
$$
 Here  $f_0(\omega, t)$ denotes the integrand of the first integral on the r.h.s~of \eqref{level2}, and from the previous results one concludes that 
 $\e  (\int_{(t_m,T]}  |f_0(s-)| d[B^n]_s)^p < \infty.$ The functions   $f_1(t) = (D^{(1,0,0)} f)(t, \cdot) = (\partial_x f)(t, \cdot) $ as well as  $f_2(t)  = (\partial_y f)(t, \cdot)$
 and $f_3(t)  = (\partial_z f)(t, \cdot)$ are  bounded by our assumptions. We put
$$\hat\xi^n :=  g''(X_T^{n}) (  \partial_x X_T^{n})^2 + g'(X_T^{n}) \partial_x^2  X_T^{n}.$$ 
Denoting   the solution by  $(\hat\sY, \hat\sZ)$   we get for $C(f_3 ) (T-t_m) \le \tfrac{1}{2}$ that
\equal \label{aprioriII}
&& \less \e |\hat \sY_{t_m}|^2 +  \half \e \int_{(t_m,T]}    |\hat \sZ_{s-}|^2 d[B^n]_s  \notag\\
&\le&\!\!\!\!\!C \bigg [ \e |\hat \xi^n|^2  +   \e \Big  (\int_{(t_m,T]}  |f_0( s-) |   d[B^n]_s \Big )^2      +  \e \int_{(t_m,T]} |\hat \sX_{s-}|^2   +   |\hat \sY_{s-} |^2 d[B^n]_s \bigg]. 
\tionl
Now we derive the  polynomial growth $\e | \hat \xi^n|^2 \le C \Psi^2(x) $ from the properties of $g'$ and $g''$ and from the fact that  $\e \sup_{t_m< s \le T}   |\partial^j_x   X_{s}^{n}|^p$ is bounded for $j=1,2$ under  our assumptions. Then the estimate $$  \e \Big (\int_{(t_m,T]} | f_0(s-) |   d[B^n]_s \Big )^2   \le C \Psi^4(x)$$ can be derived from Lemma  \ref{polynomial-estimates}\eqref{supY}-\eqref{Z-power},
so that Gronwall's  lemma  implies  
\equal \label{uxx-bound}
|\hat\sY_{t_m}^{t_m,x}|= |u_{xx}(t_m,x)| \le C \Psi^2(x). 
\tionl \bigskip

Finally, to show \eqref{u-xx-alpha}, one uses \eqref{level2}  and derives an inequality as in \eqref{aprioriII} but now for the difference
$\partial_x^2  Y_{t_m}^{n,t_m,x}-\partial_x^2  Y_{t_m}^{n,t_m,\bar x}.$

Before proving it, let us state the following lemma.

  \begin{lemme}\label{lem5}
    Let Assumption \ref{Xgeneral} hold. We have
  \begin{align}
&\left(\e \sup_s | Z^{n,t_m,x}_{s-}- Z^{n,t_m,\bar x}_{s-} |^p\right)^{1/p}
   \le C( \Psi(x)^2 + \Psi(\bar x)^2)|x-\bar x|,\quad  p\ge 2, \label{eq28}\\
 & \e \Big ( \int_{(t_m,T]}    |\partial_x
  Z^{n,t_m,x}_{s-} - \partial_x Z^{n,t_m,\bar x}_{s-}  |^2 d[B^n]_s \Big
  )^\frac{p}{2} \le  C (\Psi^{4p}(x) + \Psi^{4p}(\bar x))|x-\bar x|^p, \quad
  p\ge 2,  \label{eq27}\\
&  \e \left (  \int_{(t_m,T]} |\partial^2_x Z^{n,t_m,x}_{s-}|^{2} d[B^n]_s
  \right )^\frac{p}{2} \le C\Psi^{4p}(x), \quad  p \geq2,\label{eq29}
\end{align}
 for some constant  $C=C(b,\sigma,f,g,T,p,p_0)$.
\end{lemme}

\begin{proof}[ Proof of Lemma \ref{lem5}]  \eqref{eq28}: Introduce $G(t_{k+1}, x) := \mathcal{D}^n_{k+1} u^n(t_{k+1}, X^{n,t_k,x}_{t_{k+1}})$. Using relations \eqref{one-step-for-X}--\eqref{T-applied-to-un} and the bounds \eqref{ux-bound} and \eqref{uxx-bound} for $u^n_x$ and $u^n_{xx}$, respectively, one obtains
	\begin{align*}
	|G(t_{k+1}, x) - G(t_{k+1}, \bar{x})| \leq C(1+|x|^{2(p_0+1)} + |\bar{x}|^{2(p_0+1)})|x-\bar{x}|, \quad x,\bar{x} \in \R,
	\end{align*}
	uniformly in $t_{k+1}$. Since $Z^{n,t_m,x}_{t_k} = \mathcal{D}^n_{k+1} u^n(t_{k+1}, X^{n,t_k, \eta}_{t_{k+1}}) = G(t_{k+1}, \eta)$ where $\eta = X^{n,t_m,x}_{t_k}$, the previous bound yields
	\begin{align*}
	|Z^{n,t_m,x}_{t_k} - Z^{n,t_m, \bar{x}}_{t_k}| \leq C(1+ |X^{n,t_m,x}_{t_{k}}|^{2(p_0+1)} + |X^{n,t_m, \bar{x}}_{t_{k}}|^{2(p_0+1)})|X^{n,t_m,x}_{t_{k}}-X^{n,t_m,\bar{x}}_{t_{k}}| 
	\end{align*}
	uniformly for each $t_m \leq t_k < T$. Inequality \eqref{eq28} then follows by applying the Cauchy-Schwarz inequality and standard $L_p$-estimates for the process $X^n$.
	
  \eqref{eq27}: This can be shown similarly as Lemma
\ref{polynomial-estimates}-\eqref{Z-power}  considering the  BSDE for  the
difference   $\partial_x  Y_{t_m}^{n,t_m,x}-\partial_x  Y_{t_m}^{n,t_m,\bar
  x}$  instead of \eqref{level1} itself. 
  
  \eqref{eq29}: This one gets  repeating again the proof of  Lemma  \ref{polynomial-estimates}-\eqref{Z-power} but now for the BSDE \eqref{level2}.
\end{proof}

By our assumptions we have
$$ \e  |\hat\xi^{n,t_m,x} - \hat\xi^{n,t_m, \bar x}|^2 \le C(\Psi^2(x) + \Psi^2(\bar x))(1+|x|^2+|\bar x|^2) |x-\bar x|^{2\alpha},  $$   
where we use  $|x-\bar x|^2 \le C(1+|x|^2+|\bar x|^2)   |x-\bar x|^{2\alpha}.$ The term  $|x-\bar x|^2$ appears for example in the estimate of  
$(\partial_x X_T^{n,t_m,x })^2 -  (  \partial_x X_T^{n,t_m, \bar x })^2.$  To see that 
$$  \e \Big (\int_{(t_m,T]} | f_0^{t_m,x} (s-) -f_0^{t_m,\bar x}(s-) |   d[B^n]_s \Big )^2   \le C( \Psi^{10}(x) +  \Psi^{10}(\bar x) ) (1+|x|^{2} +|\bar x|^{2})  |x-\bar x|^{2\alpha},$$
we  check the terms with the highest polynomial growth.  For example, we
  have to deal with terms
like $\e \Big  (\!\int_{(t_m,T]}  | Z^{n,t_m,x}_{s-}- Z^{n,t_m,\bar x}_{s-} |
\,|\partial_xZ^{n,t_m,x}_{s-}|^{2}   d[B^n]_s \!\Big )^2\!\!$ and $\e \Big
(\!\int_{(t_m,T]} |\partial_x Z^{n,t_m,x}_{s-}|^2- |\partial_x Z^{n,t_m, \bar
  x}_{s-}|^{2}   d[B^n]_s \Big )^2.$
We bound the first term by using \eqref{eq26} and \eqref{eq28}
 \equa && \less \e \Big  (\int_{(t_m,T]}  | Z^{n,t_m,x}_{s-}- Z^{n,t_m,\bar x}_{s-} | \,|\partial_xZ^{n,t_m,x}_{s-}|^{2}   d[B^n]_s \Big )^2 \\
  &\le& (  \e \sup_s | Z^{n,t_m,x}_{s-}- Z^{n,t_m,\bar x}_{s-} |^4 )^\half  \Big  (\e    \Big  (\int_{(t_m,T]}  |\partial_x Z^{n,t_m,x}_{s-}|^{2}   d[B^n]_s \Big )^4  \Big )^\half\\
   &\le& C( \Psi^4(x) +  \Psi^4(\bar x) )|x-\bar x|^2  \Psi^4(x).
   \tion
  We bound the second term by using \eqref{eq26} and \eqref{eq27}
 \equa && \less \e \Big  (\int_{(t_m,T]} |\partial_x Z^{n,t_m,x}_{s-}|^2- |\partial_x Z^{n,t_m, \bar x}_{s-}|^{2}   d[B^n]_s \Big )^2 \\
  &\le&  C \e  \int_{(t_m,T]} |\partial_x Z^{n,t_m,x}_{s-}|^2+ |\partial_x Z^{n,t_m, \bar x}_{s-}|^{2}   d[B^n]_s   \int_{(t_m,T]} |\partial_x Z^{n,t_m,x}_{s-}- \partial_x Z^{n,t_m, \bar x}_{s-}|^{2}   d[B^n]_s  \\
 &\le& C( \Psi^2(x) +  \Psi^2(\bar x) )  ( \Psi^8(x) +  \Psi^8(\bar x) )|x-\bar x|^2 \\
   &\le&  C( \Psi^{10}(x) +  \Psi^{10}(\bar x) ) (|x|^{2-2\alpha} +|\bar x|^{2-2\alpha})  |x-\bar x|^{2\alpha} \\
   &\le& C( \Psi^{10}(x) +  \Psi^{10}(\bar x) ) (1+|x|^{2} +|\bar x|^{2})  |x-\bar x|^{2\alpha},
 \tion

While all the other terms can be easily estimated using the results we have obtained already, for   
$$\e \Big ( \int_{(t_m,T]}    |(f_3^{t_m,x} (s-)  -f_3^{t_m,\bar x}(s-))\partial^2_x Z^{n,t_m,x}_{s-} |d[B^n]_s  \Big)^2  \le C( \Psi^{12}(x) +  \Psi^{12}(\bar x) ) (1+|x|^{2} +|\bar x|^{2})  |x-\bar x|^{2\alpha}$$ 
we need the bound \eqref{eq29}.

 The result follows then from Gronwall's lemma. 

\begin{rem} \label{F-n-property}
 Under Assumption \ref{Xgeneral}  we conclude that by Proposition \ref{u-discrete} there exists a constant  $C = C(b,\sigma, f,g, T,p,p_0) > 0$  such that
 \equal\label{D-ux-difference}
  |u^n(t_m,x) - u^n(t_m,\bar x) |  &\le &  C( 1+ \Psi(x) + \Psi(\bar x))|x-\bar{x}|,  \notag \\
 |\mathcal{D}^n_{m+1} u^n(t_{m+1},X^{n,t_m,x}_{t_{m+1}})-\mathcal{D}^n_{m+1} u^n(t_{m+1},X^{n,t_m,x}_{t_{m+1}})|  &\le & C( 1+ \Psi^2(x) + \Psi^2(\bar x))|x-\bar{x}|, \notag \\
 |u_x^n(t_m,x) - u_x^n(t_m,\bar x) |  &\le &   C( 1+ \Psi^2(x) + \Psi^2(\bar x))|x-\bar{x}|, \notag \\
 |\partial_x \mathcal{D}^n_{m+1} u^n(t_{m+1}, X^{n,t_m,x}_{t_{m+1}}) - \partial_x \mathcal{D}^n_{m+1} u^n(t_{m+1}, X^{n,t_m,{\bar x}}_{t_{m+1}})| 
 &\leq& C(1+ \hat{\Psi}(x) + \hat{\Psi}(\bar x))|x-\bar{x}|^{\alpha},  \notag \\
  |\partial_x \mathcal{D}^n_{m+1} u^n(t_{m+1}, X^{n,t_m,x}_{t_{m+1}})| &\le & C(1+ \Psi^2(x)),
 \tionl 
uniformly in $m = 0,1, \dots, n-1$, where
\begin{align} \label{def-hat-psi}
\hat{\Psi}(x) := 1+|x|^{6p_0+8}.
\end{align}
 In addition, for $\partial_x F^n(t_{m+1}, x):=\partial_xf(t_{m+1}, x, u^n(t_m, x),  \mathcal{D}^n_{m+1} u^n(t_{m+1},X^{n,t_m,x}_{t_{m+1}}))$ we have 
\begin{align}\label{D-Fx-difference} 
|\partial_x F^n(t_{m+1}, x) - \partial_x F^n(t_{m+1}, {\bar x}) |\leq C(1+ \hat{\Psi}(x) + \hat{\Psi}(\bar x))|x-\bar{x}|^{\alpha}
\end{align}
uniformly in $m = 0,1, \dots, n-1$. The latter inequality follows from the assumption that the partial derivatives of $f$ are bounded and 
 Lipschitz continuous w.r.t.~the spatial variables, from estimates proved in  Proposition \ref{u-discrete} and from those stated in \eqref{D-ux-difference} above.  

From the calculations it can be seen that in general Assumption \ref{Xgeneral} 
can not be weakened if one needs    $ \partial_x F^n(t_{m+1}, x)$   to be locally $\alpha$-H\"older continuous.

\end{rem}

 \end{proof}

\section{Technical results and estimates}  \label{5}
In this section we collect  some facts which are needed for the proofs of our results. We start with properties of the stopping times  used to construct a random walk.
\begin{lemme}[Proposition 11.1 \cite{Walsh}, Lemma A.1 \cite{GLL}]   \label{B-difference}
For all $0 \leq k \leq m \leq n$ and $p > 0$, it holds for $h = \tfrac{T}{n}$ and  $\tau_k$ defined in \eqref{tau-k} that
\begin{enumerate}[(i)]
\item $\e \tau_k  = kh$, \label{tauk-expectation}
\item $\e |\tau_1 |^p \leq  C(p) h^p$, \label{tau1-expectation}
\item $\e | B_{\tau_k} - B_{t_k}|^{2p} \leq  C(p) \e |\tau_k - t_k|^p
  \leq  C(p) (t_k h)^{\frac{p}{2}}. \label{B-difference-estimatep}$
\end{enumerate}
\end{lemme}

The next lemma lists some estimates concerning the diffusion $X$  defined by \eqref{tilde-X} and its  discretization \eqref{X-n-with-tilde}, where we assume that $B$ and $\tilde B$  are connected as in \eqref{independent-B}.

\begin{lemme}\label{X-Lemma} 
 \label{discrete-time-sup}  
Under Assumption \ref{hypo2} on $b$ and $\sigma$ it holds for  $p \geq 2$ that
there exists a constant $C=C(b,\sigma,T,p)  >0$ such that
\begin{enumerate}[(i)]
\item $ \e\big|X^{s,y}_{T} - X^{t,x}_T\big|^p \leq  C ( |y-x|^p + |s-t|^{\frac{p}{2}}), \quad x,y \in \R, \, s,t \in [0,T],     $\label{Xsx-Xty} 
\item $\label{supX} \ti \e  \sup_{\ti \tau_l \wedge t_{m} \le r \le \ti \tau_{l+1} \wedge t_{m}} |\ti X^{t_k,x}_{t_k+r}-  \ti X^{t_k,x}_{t_k+ \ti \tau_l \wedge
    t_{m}}|^{p} \le  C h^{\frac{p}{4}}, \,\, 0\le k\le n, \, 0 \le  l \le n-k-1, \,  0  \leq m \le n-k,$  
\item    $ \e|\nb  X^{s,y}_T - \nb  X^{t,x}_T |^p \le  C( |y-x|^p + |s-t|^{\frac{p}{2}}), \quad x,y \in \R, \, s,t \in [0,T],    $   \label{nablaX-diff} 
\item $  \e \sup_{0 \leq l \leq m} \big|\nabla X^{n,t_k,x}_{t_k+t_l}\big|^{p} \leq C,   \quad 0 \leq k \leq n, \,  0 \le m \le n-k, $ \label{supcX}
\item  $\ti \e\big|\ti X^{t_k,x}_{t_k+ t_m}-\ti \cX^{\tau_k,y}_{\tau_k +\ti
    \tau_m}\big|^p \le C ( |x - y|^p+  h^{\frac{p}{4}}), \quad 0 \leq k \leq n,\, 0  \leq m \le n-k,$  \label{XandcX}
\item  $\ti \e | \nabla \ti X_{t_k+ t_m }^{t_k, x} -    \nabla \ti \cX_{\tau_k
    +\ti \tau_m}^{\tau_k,y}|^p \le C( |x - y |^p+  h^{\frac{p}{4}}), \quad 0 \leq k \leq n, \,  0 \leq m  \le n-k.$  \label{nablaX-nablacX}

\end{enumerate}
\end{lemme} 
\begin{proof}
\noindent \eqref{Xsx-Xty}: This  estimate  is  well-known. \\
\eqref{supX}: For the stochastic integral we use the  inequality of BDG  and then, since $b$ and $\sigma $ are bounded, we get 
by Lemma \ref{B-difference} \eqref{tau1-expectation} that
\equa
&& \less \ti \e  \sup_{\ti \tau_l \wedge t_{m} \le r \le \ti \tau_{l+1} \wedge
    t_{m}} |\ti X^{t_k,x}_{t_k+r}-  \ti X^{t_k,x}_{t_k+ \ti \tau_l \wedge
    t_{m}}|^{p} \\
& \le &  C(p)(\norm{b}^p_{\infty} \ti\e | \ti \tau_{l+1}- \ti \tau_l  |^p + \norm{\sigma}^p_{\infty} \e | \ti \tau_{l+1}- \ti \tau_l 
|^{\frac{p}{2}}) \le C(b,\sigma, T,p) \, h^{\frac{p}{2}}.
\tion 
\noindent \eqref{nablaX-diff}:  This can be easily seen because the process  $(\nb  X^{s,y}_r)_{r \in [s,T]} $   solves  the linear SDE  \eqref{nabla-X} with bounded coefficients. \\
\noindent \eqref{supcX}:  The  process solves \eqref{nabla cX}. The estimate follows from 
the inequality of BDG and Gronwall's lemma.\\
\eqref{XandcX}: Recall that from \eqref{Xn}  and \eqref{X-n-with-tilde} we have
$$\ti \cX^{\tau_k,y}_{\tau_k+\ti \tau_m} = \ti X^{n,t_k,y}_{t_k+t_m} = y + \int_{(0, t_m]} b(t_k+ r, \ti X^{n,t_k,y}_{t_k+r-}) d[\ti B^n, \ti B^n]_r + \int_{(0, t_m]}  \sigma(t_k+r,\ti X^{n,t_k,y}_{t_k+r-})
        d \ti B^n_{r},$$
and  $\ti X^{t_k,x}_{t_k+t_m} $ is given by       
$$\ti X^{t_k,x}_{t_k+t_m} =    x    + \int_0^{t_m} b(t_k+r, \ti X^{t_k,x}_{t_k+r}) dr + \int_0^{t_m}  \sigma(t_k+r,\ti X^{t_k,y}_{t_k+r})
        d \ti B_r.$$
To compare the stochastic integrals of the previous two equations we use the relation
$$ \int_{(0, t_m]}  \sigma(t_k+r,\ti X^{n,t_k,y}_{t_k+r-}) d \ti B^n_r =   \int_0^\infty   \sum_{l=0}^{m-1}  \sigma(t_{k+l+1},\ti X^{n,t_k,y}_{t_{k+l}}) \ind_{(\ti \tau_l, \ti \tau_{l+1}]}(r)  d \ti B_r.$$
We define an 'increasing' map   $i(r) := t_{l+1} $ for  $r$ in $(t_l,t_{l+1}]$  and a 'decreasing' map $d(r)   := t_{l} $ for $(t_l,t_{l+1}]$  and split the differences as follows (using Assumption \ref{hypo2}-(iii)
for the coefficient $b$)
\equal \label{iteration}
&& \less  \ti \e\big|\ti X^{t_k,x}_{t_k+t_m}- \ti X^{n,t_k,y}_{t_k+t_m}\big|^p  \notag \\ 
&\le & C(b,p) \left ( |x - y|^p 
   +   \ti \e \int_0^{t_m} |r- i(r)|^{\frac{p}{2}}  +| \ti X^{t_k,x}_{t_k+r}- \ti X^{t_k,x}_{t_k+d(r)} |^p + | \ti X^{t_k,x}_{t_k+d(r)}- \ti X^{n,t_k,y}_{t_k+d(r)} |^p  dr \right)\notag\\
&& +C(p)   \ti \e | \int_{t_{m} \wedge \ti \tau_{m}}^{t_{m}} \sigma(t_k+r,\ti X^{t_k,x}_{t_k+r}) d \ti B_r|^p \notag \\
&&+ C(p)  \ti \e | \int_{t_{m} \wedge \ti \tau_{m}}^{\ti \tau_{m}}  \sum_{l=0}^{ m-1}  \sigma(t_{k+l+1},\ti X^{n,t_k,y}_{t_{k+l}}) 
\ind_{(\ti \tau_l, \ti \tau_{l+1}]}(r)  d \ti B_r|^p\notag \\
&& +   C(p)  \ti \e | \int_0^{t_{m} \wedge \ti \tau_{m}}  \!\!\!\! \sigma(t_k+r,\ti X^{t_k,x}_{t_k+r})  -  \sum_{l=0}^{m-1}  \sigma(t_{k+l+1},\ti X^{n,t_k,y}_{t_{k+l}}) 
\ind_{(\ti \tau_l, \ti \tau_{l+1}]}(r)     d \ti B_r|^p.  
\tionl        
We estimate the terms on the r.h.s as follows: by standard estimates for SDEs with bounded coefficients one has that
\equa
  \ti \e \int_0^{t_m} |r- i(r)|^{\frac{p}{2}}   +| \ti X^{t_k,x}_{t_k+r}- \ti
  X^{t_k,x}_{t_k+d(r)} |^p dr \le  C(b,\sigma,T,p) h^\frac{p}{2}.
\tion 
By the BDG inequality, the fact that $\sigma$ is bounded  and Lemma \ref{B-difference} we conclude that 
\equa&& \less \ti \e \bigg | \int_{t_{m} \wedge \ti \tau_{m}}^{t_{m}} \sigma(t_k+r,\ti X^{t_k,x}_{t_k+r}) d \ti B_r\bigg |^p + \ti \e \bigg | \int_{t_{m} \wedge \ti \tau_{m}}^{\ti \tau_{m}}  \sum_{l=0}^{m-1}  \sigma(t_{k+l+1},\ti X^{n,t_k,y}_{t_{k+l}}) 
\ind_{(\ti \tau_l, \ti \tau_{l+1}]}(r)  d \ti B_r \bigg|^p \\
&& \le  C(\sigma,p)  \|\sigma\|^p_\infty \ti \e  |\ti
\tau_{m}-t_{m}|^\frac{p}{2} \le  C(\sigma,p) ( t_{m} h)^\frac{p}{4}. 
\tion 
Finally, by the BDG inequality
\equa && \less \ti \e \bigg | \int_0^{t_{m} \wedge \ti \tau_{m}}  \!\!\!\! \sigma(t_k+r,\ti X^{t_k,x}_{t_k+r})  -  \sum_{l=0}^{m-1}  
\sigma(t_{k+l+1},\ti X^{n,t_k,y}_{t_{k+l}}) 
\ind_{(\ti \tau_l, \ti \tau_{l+1}]}(r)     d \ti B_r \bigg|^p \\
&\le & C(p) \ti \e \bigg ( \int_0^{t_{m}}  \sum_{l=0}^{m-1}
|\sigma(t_k+r,\ti X^{t_k,x}_{t_k+r})-\sigma(t_{k+l+1},\ti X^{n,t_k,y}_{t_{k+l}})|^2 \ind_{(\ti \tau_l, \ti \tau_{l+1}]}(r)  dr\bigg )^\frac{p}{2} \\
&\le &  C(\sigma,p)\ti \e \bigg ( \sum_{l=0}^{m-1}\int_{\ti \tau_l \wedge t_{m}}^{\ti \tau_{l+1} \wedge t_{m}}\!\!\!\! 
|\ti \tau_{l+1}-t_{l+1}|^\frac{p}{2} + |\ti \tau_l -t_{l+1} |^\frac{p}{2} + |\ti X^{t_k,x}_{t_k +r}- \ti X^{t_k,x}_{t_k+\ti \tau_l \wedge t_{m}}|^p \\ 
&&  \hspace{20em} + |\ti X^{t_k,x}_{t_k+ \ti \tau_l \wedge t_{m}}-\ti X^{n,t_k,y}_{t_{k+l}} |^p dr\bigg )  \\
&\le&   C(\sigma,T,p) \Big ( h^\frac{p}{2}  + \max_{1\le l<m} ( \ti \e  |\ti \tau_l -t_l|^p)^\half 
 +  \max_{0\le l<m}  ( \ti \e  \sup_{\ti \tau_l \wedge t_{m} \le r \le \ti \tau_{l+1} \wedge t_{m}} |\ti X^{t_k,x}_{t_k+r}- 
 \ti X^{t_k,x}_{t_k+ \ti \tau_l \wedge t_{m}}|^{2p} )^\half \\
&&  +  \ti \e  \sum_{l=0}^{m-1}   |\ti X^{t_k,x}_{t_k+\ti \tau_l \wedge t_{m}} -\ti X^{n,t_k,y}_{t_{k+l}} |^p (\ti \tau_{l+1} - \ti \tau_l ) \Big ).
 \tion  
Moreover, since  $\ti \tau_{l+1} - \ti \tau_l$ is independent from  $|\ti X^{t_k,x}_{t_k+ \ti \tau_l \wedge t_{m}} -\ti X^{n,t_k,y}_{t_k+t_l} |^p$ we get by Lemma  \ref{B-difference}-\eqref{tauk-expectation} 
\equa
&& \less  \ti \e  \sum_{l=0}^{m-1}  |\ti X^{t_k,x}_{t_k+ \ti \tau_l \wedge t_{m}} -\ti X^{n,t_k,y}_{t_{k+l}} |^p   (\ti \tau_{l+1} - \ti \tau_l ) \\
&=&
 \ti \e   \sum_{l=0}^{m-1}  |\ti X^{t_k,x}_{t_k+\ti \tau_l \wedge t_{m}} -\ti X^{n,t_k,y}_{t_{k+l}} |^p  (t_{l+1} - t_l )\\
&\le&  C(T,p) \Big (   \ti \e    \int_0^{t_m}  |\ti X^{t_k,x}_{t_k+d(r)} -\ti X^{n,t_k,y}_{t_k+d(r)} |^p dr + \max_{0\le l<m}  \ti \e  |\ti X^{t_k,x}_{t_k+ \ti \tau_l \wedge t_{m}}- \ti X^{t_k,x}_{t_k+t_l}|^p \Big ).
 \tion
Using  Lemma \ref{B-difference}-\eqref{B-difference-estimatep}  one concludes similarly as in the proof of \eqref{supX} that $ \ti \e  |\ti
X^{t_k,x}_{t_k+ \ti \tau_l \wedge t_{m}}- \ti X^{t_k,x}_{t_k+t_l}|^p\le C(b,\sigma,T,p) h^\frac{p}{4} .$   
 Then  \eqref{iteration}  combined with the above estimates  implies that
\equa
\ti \e\big|\ti X^{t_k,x}_{t_k+t_m}- \ti X^{n,t_k,y}_{t_k+t_m}\big|^p 
\le C(b,\sigma,T,p) \Big ( |x - y|^p + h^\frac{p}{4} +    \ti \e    \int_0^{t_m}  |\ti X^{t_k,x}_{t_k+ d(r)} -\ti X^{n,t_k,y}_{t_k+d(r)} |^p dr \Big ).
 \tion
Then Gronwall's lemma yields  
$$
\ti \e\big|\ti X^{t_k,x}_{t_k+t_m}- \ti X^{n,t_k,y}_{t_k+t_m}\big|^p 
\le   C(b,\sigma,T,p)  ( |x - y|^p + h^\frac{p}{4}). $$
\noindent \eqref{nablaX-nablacX}:  We have 
\equal  \label{nabla cX}
 \nabla \ti X^{n,t_k, y}_{t_k+ t_m}
&=& 1 + \!\!\int_{(0,t_m]} b_x (t_k+ r,X^{n, t_k, y}_{t_k+r-})   \nabla \ti X^{n,t_k, y}_{t_k+r-}  d[\ti B^n,\ti B^n]_r  \!  \notag\\
&&+ \!\int_{(0,t_m]}  \sigma_x(t_k+ r, \ti X^{n,t_k,y}_{t_k+r-}) \nabla  \ti X^{n,t_k,  y}_{t_k+r-} d \ti B^n_r
\tionl       
and  
\equal \label{nabla X}
\nabla \ti X^{t_k,x}_{t_k+t_m} =  1   + \int_0^{t_m} b_x(t_k+ r, \ti X^{t_k,x}_{t_k+r})   \nabla \ti X^{t_k,x}_{t_k+r}dr 
                     + \int_0^{t_m}  \sigma_x(t_k+ r,\ti X^{t_k,x}_{t_k+r}) \nabla \ti X^{t_k,x}_{t_k+r}d \ti B_{r}.
\tionl                     
We may proceed similarly  as in  \eqref{XandcX} but this time the coefficients are not bounded but have linear growth. Here  one uses that the integrands are bounded in any $L_p(\PP).$                
        \end{proof} 
   
Finally, we estimate  the difference between the continuous-time Malliavin weight and its dis\-crete-time counterpart. 
        
\begin{lemme} \label{N-difference-estimate} Let $B$ and $\tilde B$ be connected via  \eqref{independent-B}. Under  Assumption \ref{hypo2} it holds that 
$$ \ti \e|\ti N^{t_k}_{t_m}\sigma(t_k,X_{t_k})
-\ti N^{n,\tau_k}_{\ti \tau_m}
         \sigma(t_{k+1},\cX_{\tau_k})|^2 \leq   C(b,\sigma,T,\delta) \frac{ |X_{t_k} - \cX_{\tau_k}|^2+  h^\frac{1}{2}}{(t_m-t_k)^{\frac{3}{2}}}, \quad m=k+1,...,n. $$

\end{lemme}

\begin{proof} 
For   $N^{n,\tau_k}_{\ti \tau_m}$   and  $N^{t_k}_{t_{m}}$ given by   \eqref{N-cont} and    \eqref{N-discrete},  respectively,  we introduce the 
notation 
\equa
\ti N^{t_k}_{t_m} \sigma(t_k,X_{t_k})=: \frac{1}{t_{m-k}} \int_{0}^{t_{m-k}} a_{t_k+s} d\ti B_{s} \quad \text{and } \quad  \ti N^{n,\tau_k}_{\ti \tau_m}
         \sigma(t_{k+1},\cX_{\tau_k})=:  \frac{1}{t_{m-k}} \int_{0}^{\ti \tau_{m-k}}  a_{\tau_k+s}^n   d\ti B_{s} 
                         \tion
with 
\equa  a_{t_k+s}:=    \nabla \ti X_{t_k+s}^{t_k, X_{t_k}}\frac{\sigma(t_k,X_{t_k})}{\sigma(t_k\!+\!s,\ti X_{t_k+s}^{t_k,X_{t_k}})}  
\,\, \text{ and } \,\,
a_{\tau_k+s}^n   :=  \sum_{\ell=1}^{m-k} \nabla \ti \cX_{\tau_k + \ti \tau_{\ell-1}}^{\tau_k,\cX_{\tau_k}} \frac{\sigma(t_{k+1},\cX_{\tau_k})}{\sigma(t_{k+\ell},\ti \cX_{\tau_k+ \ti \tau_{\ell-1}}^{\tau_k,\cX_{\tau_k}})}  \ind_{s
    \in (\ti \tau_{\ell-1}, \ti \tau_\ell]}.
 \tion
By the inequality of BDG,  \equa
&& \less  (t_m-t_k)^2 \ti \e|\ti N^{t_k}_{t_m} \sigma(t_k,X_{t_k})-\ti N^{n,\tau_k}_{\ti \tau_m} \sigma(t_{k+1},\cX_{\tau_k})|^2 \\
	   & = &  \ti \e \Big|\int_{0}^{t_{m-k}} a_{t_k+s} d\ti B_s -\int_{0}^{\ti \tau_{m-k}} a^n_{\tau_k+s} d\ti B_{s}\Big|^2\\
& =&  \ti \e \int_{0}^{t_{m-k}\wedge  \ti \tau_{m-k}} (a_{t_k+s} - a^n_{\tau_k+s})^2 ds  + \ti \e \int_{0}^{\infty} a_{t_k+s}^2 \ind_{(\ti \tau_{m-k},t_{m-k}]}(s)  ds\\
&&  + \ti \e \int_{0}^{\infty} (a^n_{\tau_k+s})^2 \ind_{(t_{m-k},\ti \tau_{m-k}]}(s)  ds\\
& \leq&   \sum_{\ell=1}^{m-k} \pr{\ti \e \sup_{s \in [0, t_{m-k}] \cap (\ti \tau_{\ell-1}, \ti \tau_{\ell}]}\big|a_{t_k+s} - a^n_{ \tau_k +\ti \tau_{\ell}}\big|^4}^{\frac{1}{2}} (\ti \e |\ti \tau_{\ell} -\ti \tau_{\ell-1}|^2)^\half  \\
&& + \pr{\ti \e \sup_{s \in [0, t_{m-k}]} |a_{t_k+s}|^4  + \ti \e \max_{1 \leq \ell \leq m-k} |a^n_{\tau_k +\ti \tau_{\ell}}|^4}^{\frac{1}{2}} (\ti \e |t_{m-k} -\ti \tau_{m-k}|^2 )^{\frac{1}{2}}.
\tion 
The assertion follows then  from Lemma \ref{B-difference} and from the estimates
\begin{align}
\ti \e \sup_{s \in [0,t_{m-k}]\cap [\ti \tau_{\ell-1}, \ti \tau_{\ell}]}
  |a_{t_k+s}-  a^n_{\tau_k+ \ti \tau_\ell}|^4 & \le   C(b,\sigma, T,\delta)  (|X_{t_k} - X^n_{t_k}|^4+  h) \label{a-difference}  \\
\ti \e \sup_{s \in [0,t_{m-k}]} | a_{t_k+s}|^4 + \ti \e \max_{1 \leq \ell \leq m-k} |a^n_{\tau_k+\ti \tau_{\ell}}|^4 & \leq 2 \|\sigma\|_{\infty}^4\delta^{-4}.  \label{a-bounded}
\end{align}
 
So it remains to show these inequalities. We put
$$    \ti K^{t_k}_{t_k+s} := \frac{\sigma(t_k,X_{t_k})}{\sigma(t_k+s,\ti X_{t_k+s}^{t_k,X_{t_k}})}  \quad \text{ and }  \quad \ti K^{n,\tau_k}_{\tau_k+\ti \tau_{\ell-1}} 
:= \frac{\sigma(t_{k+1}, \cX_{\tau_k})}{\sigma(t_{k+\ell}, \ti \cX^{\tau_k, \cX_{\tau_k}}_{\tau_k +\ti \tau_{\ell-1}})}$$
and notice that by Assumption  \ref{hypo2} both expressions are bounded by $ \norm{\sigma}_{\infty}\delta^{-1}.$ 
To show \eqref{a-difference}  let us split $a_{t_k+s}-  a^n_{\tau_k+ \ti \tau_\ell}$ in the following way:
  \begin{align*}
 a_{t_k+s}-  a^n_{\tau_k+ \ti \tau_\ell}
 = & \,  \ti K^{t_k}_{t_k+s} (\nabla \ti X_{t_k+s}^{t_k, X_{t_k}} -\nabla \ti X_{t_k + t_{\ell-1}}^{t_k, X_{t_k}})
        +  \nabla \ti X_{t_k + t_{\ell-1}}^{t_k, X_{t_k}} ( \ti K^{t_k}_{t_k+s}  -   \ti K^{t_k}_{t_k +t_{\ell-1}}) \\
    &+  \ti K^{t_k}_{t_k +t_{\ell-1}} ( \nabla \ti X_{t_k+ t_{\ell-1} }^{t_k, X_{t_k}} -    \nabla \ti \cX_{\tau_k +\ti \tau_{\ell-1}}^{\tau_k,\cX_{\tau_k}} )   
        +  \nabla \ti \cX_{\tau_k +\ti \tau_{\ell-1}}^{\tau_k,\cX_{\tau_k}}  ( \ti K^{t_k}_{t_k+ t_{\ell-1}} - \ti K^{n,\tau_k}_{\tau_k +\ti \tau_{\ell-1}}).
  \end{align*}
Then 
\equa 
&& \less \ti \e  \sup_{s \in [\ti \tau_{\ell-1} \wedge t_{m-k}, \ti \tau_{\ell} \wedge t_{m-k}]}  |  \ti K^{t_k}_{t_k+s} (\nabla \ti X_{t_k+s}^{t_k, X_{t_k}} -\nabla \ti X_{t_k + t_{\ell-1}}^{t_k, X_{t_k}})|^4  \\ 
&\le& \norm{\sigma}^4_{\infty}\delta^{-4} \ti \e  \sup_{s \in [\ti \tau_{\ell-1} \wedge t_{m-k}, \ti \tau_{\ell} \wedge t_{m-k}]}  |  \nabla \ti X_{t_k+s}^{t_k, X_{t_k}} -\nabla \ti X_{t_k + t_{\ell-1}}^{t_k, X_{t_k}}|^4   
\le   C(b,\sigma,T,\delta) h
\tion  
 since one can show similarly to Lemma  \ref{X-Lemma}-\eqref{supX} that 
  $$\ti \e \sup_{s \in [\ti \tau_{\ell-1} \wedge t_{m-k}, \ti \tau_{\ell} \wedge t_{m-k}]}  |  \nabla \ti X_{t_k+s}^{t_k, X_{t_k}} -\nabla \ti X_{t_k + t_{\ell-1}}^{t_k, X_{t_k}}|^4 \le  C(b,\sigma,T,\delta) h. $$ 
  
Notice that $\nabla \ti X_t^{t_k, X_{t_k}}$      and    $\nabla \ti \cX_{\tau_m}^{\tau_k,\cX_{\tau_k}} $ solve the linear SDEs   \eqref{nabla X}  and \eqref{nabla cX}, respectively. 
Therefore,    
\equal \label{bound} \ti \e \sup_{s \in
      [0,t_{m-k}]}|\nabla \ti X_{t_k+s}^{t_k, X_{t_k}}|^p \le  C(b,\sigma,T,p) \quad \text{ and } \quad 
      \ti\e  \max_{0 \leq \ell \leq m-k}| \nabla \ti \cX_{\ti \tau_{\ell}+\tau_k}^{\tau_k,\cX_{\tau_k}}|^p\leq C(b,\sigma,T,p).
      \tionl  
 For the second term we get      
\equa 
&& \less \ti \e  \sup_{s \in [\ti \tau_{\ell-1} \wedge t_{m-k}, \ti \tau_{\ell} \wedge t_{m-k}]} | \nabla \ti X_{t_k + t_{\ell-1}}^{t_k, X_{t_k}} ( \ti K^{t_k}_{t_k+s}  -   \ti K^{t_k}_{t_k +t_{\ell-1}}) |^4 \\
&\le&  C(\sigma, \delta) (\ti \e | \nabla \ti X_{t_k + t_{\ell-1}}^{t_k, X_{t_k}}|^8 )^\half     
          (\ti \e  \sup_{s \in [\ti \tau_{\ell-1} \wedge t_{m-k}, \ti \tau_{\ell} \wedge t_{m-k}]} (| t_\ell-s|^4 + | \ti X_{t_k+s}^{t_k,X_{t_k}}-\ti X_{t_k+t_\ell}^{t_k,X_{t_k}}|^8)^\half \\
 &\le &   C(b,\sigma,T,\delta) h.        
\tion 
For the  third term  Lemma  \ref{X-Lemma}-\eqref{nablaX-nablacX} implies that 
$$ \ti \e |\ti K^{t_k}_{t_k +t_{\ell-1}} ( \nabla \ti X_{t_k+ t_{\ell-1} }^{t_k, X_{t_k}} -    \nabla \ti \cX_{\tau_k +\ti \tau_{\ell-1}}^{\tau_k,\cX_{\tau_k}} )   |^4 \le   C(b,\sigma,T) \norm{\sigma}^4_{\infty}\delta^{-4} (| X_{t_k} - \cX_{\tau_k}|^4+  h). $$
The last term we estimate similarly to the second one, 
 \equa
&&\less  \ti \e |\nabla \ti \cX_{\tau_k +\ti \tau_{\ell-1}}^{\tau_k,\cX_{\tau_k}} ( \ti K^{t_k}_{t_k+ t_{\ell-1}} - \ti K^{n,\tau_k}_{\tau_k +\ti \tau_{\ell-1}})|^4\\ 
& \leq&  C(\sigma, \delta)(\ti \e | \nabla \ti \cX_{\tau_k +\ti \tau_{\ell-1}}^{\tau_k,\cX_{\tau_k}} |^8 )^\half
 ( |X_{t_k} - \cX_{\tau_k}|^{8} + \ti \e | \cX^{\tau_k, \cX_{\tau_k}}_{\tau_k +\ti \tau_{\ell-1}}-\ti X_{t_k+t_{\ell-1}}^{t_k,X_{t_k}}|^8)^\half    \\
&\leq& C(b,\sigma,T, \delta) (|X_{t_k} - \cX_{\tau_k}|^4 + h).
\tion 
To see \eqref{a-bounded} use the estimates \eqref{bound}.
\end{proof}

We close this  section with  estimates concerning the effect of $T_{_{m,{\pm}}}$  and the  discretized Malliavin derivative $\mathcal{D}^n_k$ (see Definition \ref{T-m}) on 
$X^n.$
\begin{lemme} \label{nabla and Malliavin} Under Assumption \ref{hypo2}, and  for  $p\ge 2,$ we have
\begin{enumerate}[(i)]
\item \label{nM-i} \quad $\e |X^n_{t_l}- T_{_{m,{\pm}}}X^n_{t_l}|^p \leq  C(b,\sigma,T,p) h^{\frac{p}{2}}, \quad  1 \leq l, m \leq n$,
\item \label{nM-iv} \quad $\e\left | \nabla X^{n,t_k,X^n_{t_k}}_{t_m}  - \dfrac{\mathcal{D}^n_{k+1}X^n_{t_m} }{\sigma(t_{k+1},X^n_{t_k})} \right |^p 
  \le  C(b,\sigma,T,p) h^{\frac{p}{2}}, \quad 0 \le k < m \le n.$
\item \label{nM-iii} \quad $\e |\mathcal{D}^n_k X^n_{t_m} |^p \le C(b,\sigma,T,p), \quad0 \le k \leq m \leq n$.
\end{enumerate}
\end{lemme}
\begin{proof}
\eqref{nM-i} 
By definition,  $T_{_{m, \pm}}X^n_{t_l} = X^n_{t_l}$ for $l \leq m -1, $  and 
 for $l\ge m$ we have
  \equa
T_{_{m, \pm}}X^n_{t_l} &=& X^n_{t_{m-1}} + b(t_m,X^n_{t_{m-1}})h \pm \sigma(t_m,X^n_{t_{m-1}}) \sqrt{h} \\
&& + h     \sum_{j=m+1}^{ {l}}
    b(t_j, T_{_{m,\pm}}X^n_{t_{j-1}})+ \sqrt{h} \sum_{j=m+1}^{ {l}}
    \sigma(t_{ j}, T_{_{m,\pm}}X^n_{t_{j-1}})\ep_j.
\tion
By the  properties of $b$ and $\sigma$  and thanks to the inequality of Burkholder-Davis-Gundy  and  H\"older's inequality we see that
  \begin{align*}
& \e|X^n_{t_l}- T_{_{m, \pm}}X^n_{t_l}|^p\\ & \leq C(p) \Big ( \e\big|\sigma(t_m,X^n_{t_{m-1}}) \sqrt{h}(1 \pm\varepsilon_m)\big|^p+ h^{p} \e \Big| \sum_{j = m+1 }^l \big(b(t_j, X^n_{t_{j-1}}) - b(t_j, T_{_{m,\pm}}X^n_{t_{j-1}})\big)\Big|^p\\
& \quad \quad + h^{\frac{p}{2}} \e \Big |\sum_{j=m+1}^l \big(\sigma(t_j, X^n_{t_{j-1}}) - \sigma(t_j, T_{_{m,\pm}} X^n_{t_{j-1}})\big)^2 \Big |^{\frac{p}{2}} \Big )\\
& \leq C(p) \Big ( \norm{\sigma}_{\infty}^p h^{\frac{p}{2}} + h (\norm{b_x}_{\infty}^p t_{l-m}^{p-1} + \norm{\sigma_x}_{\infty}^pt_{l-m}^{{\frac{p}{2}}-1}) \sum_{j=m+1}^l \e|X^n_{t_{j-1}}- T_{_{m,\pm}}X^n_{t_{j-1}}|^p \Big ).
\end{align*}
It remains to apply Gronwall's lemma. 

\noindent $\eqref{nM-iv}$  By the inequality of Burkholder-Davis-Gundy (BDG) and H\"older's inequality, 
\begin{align*} 
\e\!\left | \nabla X^{n,t_k,X^n_{t_k}}_{t_m} - \frac{  \mathcal{D}^n_{k+1}
  X^n_{t_{m}}}{\sigma( t_{k+1},X^n_{t_k})} \right |^p & \le C(p,T)  \bigg (  | b_x(t_{k+1}, X^n_{t_k})h + \sigma_x(t_{k+1},X^n_{t_k}) \sqrt{h} \ep_{k+1}|^p \notag \\
& \quad  + h^p \sum_{l=k+2}^{m} \e \left | b_x(t_l, X^n_{t_{l-1}})  \nabla X^{n,t_k,X^n_{t_k}}_{t_{l-1}} - b_x^{(k+1,l)} \frac{\mathcal{D}^n_{k+1} X^n_{t_{l-1}}}{\sigma( t_{k+1},X^n_{t_k})} \right|^p \nonumber\\   
& \quad +  h^{\frac{p}{2}} \!\!\!  \sum_{l=k+2}^{m} \e \bigg|\sigma_x(t_l, X^n_{t_{l-1}}) \nabla X^{n,t_k,X^n_{t_k}}_{t_{l-1}} -\sigma_x^{(k+1,l)} \frac{\mathcal{D}^n_{k+1} X^n_{t_{l-1}}}{\sigma(t_{k+1},X^n_{t_k})} \bigg|^{p} \bigg ).
\end{align*}
Since by Lemma \ref{nabla and Malliavin} \eqref{nM-i} we conclude that 
\begin{align*}
\e |b_x^{(k+1,l)} - b_x(t_l, X^n_{t_{l-1}})|^{2p} + \e |\sigma_x^{(k+1,l)} -
  \sigma_x(t_l, X^n_{t_{l-1}})|^{2p} \leq  C(b,\sigma,T,p) h^{p},
\end{align*}
and Lemma \ref{discrete-time-sup} implies that
$$ \e \sup_{k+1 \leq l \leq m} \Big |\nabla X^{n,t_k,X^n_{t_k}}_{t_{l-1}}\Big|^{2p} \leq  C(b,\sigma,T,p),$$
the assertion follows by Gronwall's lemma.  \bigskip

\noindent $\eqref{nM-iii}$ This is an immediate consequence of  \eqref{nM-i}.

 \end{proof}
 \bigskip 



\appendix



\acks
Christel Geiss would like to thank the  Erwin Schr\"odinger Institute, Vienna, for hospitality and support, where  a part of this work was written. 

%
%
%
%

\end{document}